\documentclass[10pt]{amsart}

\usepackage{enumitem,kantlipsum}

\usepackage{amsmath,amscd}
\usepackage{amssymb}

\usepackage[all]{xy}

\usepackage[T1]{fontenc}

\usepackage{times}

\usepackage[T1]{fontenc}
\usepackage{graphicx}
\usepackage[svgnames]{xcolor}
\usepackage{framed}

\usepackage{tikz}

\author{Liran Shaul}
\address{Department of Mathematics, Ben Gurion University of the Negev, Beer Sheva 84105, Israel}

\email{shlir@post.bgu.ac.il}
\newtheorem{thm}[equation]{Theorem}
\newtheorem*{thm*}{Theorem}
\newtheorem{cor}[equation]{Corollary}
\newtheorem{prop}[equation]{Proposition}
\newtheorem{lem}[equation]{Lemma}

\theoremstyle{definition}
\newtheorem{dfn}[equation]{Definition}
\newtheorem{rem}[equation]{Remark}

\newtheorem{fact}[equation]{}

\newcommand{\iso}{\xrightarrow{\simeq}}
\newcommand{\inj}{\hookrightarrow}

\newcommand{\opn}{\operatorname}
\newcommand{\cat}[1]{\operatorname{\mathsf{#1}}}

\newcommand{\mfrak}[1]{\mathfrak{#1}}

\newcommand{\mrm}[1]{\mathrm{#1}}

\newcommand{\til}[1]{\tilde{#1}}

\renewcommand{\k}{\Bbbk}

\renewcommand{\d}{\mathrm{d}}

\renewcommand{\a}{\mfrak{a}}
\renewcommand{\b}{\mfrak{b}}
\renewcommand{\c}{\mfrak{c}}
\renewcommand{\d}{\mfrak{d}}
\newcommand{\m}{\mfrak{m}}
\newcommand{\n}{\mfrak{n}}
\newcommand{\p}{\mfrak{p}}

\newcommand{\RG}[1]{\mrm{R}\Gamma_{\bar{#1}}}
\newcommand{\LL}[1]{\mrm{L}\Lambda_{\bar{#1}}}

\numberwithin{equation}{section} 

\newcommand{\cdg}{\cat{CDG}}

\newcommand{\ho}{\mrm{Ho}}

\newcommand{\cdgcont}[1][]{\cat{CDG}_{\cat{cont}{#1}}}

\newcommand{\Tau}{\pmb{\tau}}

\newcommand{\amp}{\operatorname{amp}}

\newcommand{\qqed}{\nobreak \ifvmode \relax \else
      \ifdim\lastskip<1.5em \hskip-\lastskip
      \hskip1.5em plus0em minus0.5em \fi \nobreak
      \vrule height0.75em width0.5em depth0.25em\fi}

\thanks{The research was partially supported by the Israel Science Foundation (grant no. 1346/15).}

\begin{document}

\title{Completion and torsion over commutative DG rings}

\begin{abstract}
Let $\cdgcont$ be the category whose objects are pairs $(A,\bar{\a})$,
where $A$ is a commutative DG-algebra and $\bar{\a}\subseteq \mrm{H}^0(A)$ is a finitely generated ideal,
and whose morphisms $f:(A,\bar{\a}) \to (B,\bar{\b})$ are morphisms of DG-algebras $A \to B$,
such that $(\mrm{H}^0(f)(\bar{\a})) \subseteq \bar{\b}$. 
Letting $\ho(\cdgcont)$ be its homotopy category, obtained by inverting adic quasi-isomorphisms,
we construct a functor $\mrm{L}\Lambda:\ho(\cdgcont) \to \ho(\cdgcont)$ which takes a pair $(A,\bar{\a})$ into its non-abelian derived $\bar{\a}$-adic completion. 
We show that this operation has, in a derived sense, 
the usual properties of adic completion of commutative rings, and that if $A = \mrm{H}^0(A)$ is an ordinary noetherian ring, this operation coincides with ordinary adic completion.
As an application, following a question of Buchweitz and Flenner, we show that if $\k$ is a commutative ring, and $A$ is a commutative $\k$-algebra which is $\a$-adically complete with respect to a finitely generated ideal $\a\subseteq A$,
then the derived Hochschild cohomology modules $\opn{Ext}^n_{A\otimes^{\mrm{L}}_{\k} A} (A,A)$ and the derived complete Hochschild cohomology modules $\opn{Ext}^n_{A\widehat{\otimes}^{\mrm{L}}_{\k} A} (A,A)$ coincide, without assuming any finiteness or noetherian conditions on $\k, A$ or on the map $\k \to A$.
\end{abstract}

\maketitle

\setcounter{tocdepth}{1}
\tableofcontents

\setcounter{section}{-1}

\section{Introduction}

All rings in this paper are commutative and unital.
Adic completion is a fundamental operation in commutative algebra.
In their paper \cite{BF}, Buchweitz and Flenner considered the Hochschild cohomology of adically complete noetherian algebras. They observed that in the presence of an adic topology, 
there are two versions of Hochschild cohomology: the classical Hochschild cohomology, and the complete (analytic in their terminology) Hochschild cohomology,
obtained by performing the adic completion operation on the bar resolution (see \cite[Section 3]{BF} for details). Given a flat local homomorphism $(A,\m) \to (B,\n)$ between noetherian local rings, such that the induced map $A/\m \to B/\n$ is an isomorphism, they proved that the Hochschild cohomology of $B$ over $A$ is isomorphic to $\opn{Ext}^n_{B\otimes_A B}(B,B)$ (even if $B$ is not projective over $A$), and asked if it coincides with the completed Hochschild cohomology of $B$ over $A$, which they have shown, is isomorphic to $\opn{Ext}^n_{B\widehat{\otimes}_A B}(B,B)$. Here, $B\widehat{\otimes}_A B$ is the $\b^e$-adic completion of $B\otimes_A B$, where $\b^e = \b\otimes_A 1 + 1\otimes_A \b \subseteq B\otimes_A B$.

In \cite[Corollary 4.3]{Sh1}, we answered this question positively, and, more generally, proved:
\begin{thm*}
Let $\k$ be a commutative ring, let $A$ be a commutative $\k$-algebra,
let $\a\subseteq A$ be a finitely generated ideal, and suppose that $A$ is $\a$-adically complete. Let $M$ be an $\a$-adically complete $A$-module. Assume the following:
\begin{enumerate}
\item The map $\k \to A$ is flat.
\item The ring $A$ is noetherian.
\item The composition $\k \to A \to A/\a$ is essentially of finite type.
\end{enumerate}
Then for any $n \in \mathbb{Z}$, there is an $A$-linear isomorphism
\begin{equation}\label{eqn:hoc-iso}
\opn{Ext}^n_{A\otimes_{\k} A}(A,M) \cong \opn{Ext}^n_{A\widehat{\otimes}_{\k} A}(A,M).
\end{equation}
\end{thm*}

Our proof of this result in \cite{Sh1} depended on all 3 conditions above.
One of the main results of this paper is a generalization of this result, without assuming any of the conditions (1)-(3). As a first step towards dropping this assumption, we drop the flatness assumption on the map $\k \to A$. But then, for such a homological statement to have any hope of remaining true, we must derive the enveloping algebra $A\otimes_{\k} A$,
and instead consider the derived enveloping algebra $A\otimes^{\mrm{L}}_{\k} A$.
This is no longer a commutative ring, but instead, it is a derived commutative ring. In this paper, we model derived commutative rings using commutative DG-algebras.
Using the derived enveloping algebra, the left hand side of (\ref{eqn:hoc-iso}) becomes $\opn{Ext}^n_{A\otimes^{\mrm{L}}_{\k} A}(A,M)$. These $A$-modules are called the derived Hochschild cohomology (also known as Shukla cohomology) modules of $A$ with respect to $M$.
Regarding the right hand side of (\ref{eqn:hoc-iso}), we run into a problem: it is not clear how to perform adic completion on  $A\otimes^{\mrm{L}}_{\k} A$. The majority of this paper is devoted to tackling this problem. Thus, we will explain how to perform the adic completion operation on a non-positive commutative DG-algebra $B$, with respect to a finitely generated ideal $\bar{\b} \subseteq \mrm{H}^0(B)$ (the latter is a commutative ring). We will then derive (in a homotopical sense) this non-abelian operation, and obtain a functor we call the derived adic completion functor. Applying it to the derived enveloping algebra, we will show:
\begin{thm}\label{thm:intro}
Let $\k$ be a commutative ring,
let $A$ be a commutative $\k$-algebra,
let $\a \subseteq A$ be a finitely generated ideal,
and suppose that $A$ is $\a$-adically complete.
Let $M$ be an $\a$-adically complete $A$-module.
Then for any $n \in \mathbb{Z}$, there is an $A$-linear isomorphism
\[
\opn{Ext}^n_{A\otimes^{\mrm{L}}_{\k} A}(A,M) \cong
\opn{Ext}^n_{A\widehat{\otimes}^{\mrm{L}}_{\k} A}(A,M)
\]
\end{thm}
This is contained in Theorem \ref{thm:adicHoc} below.
It is worth noting that even if $A$ is flat over $\k$, so that $A\otimes^{\mrm{L}}_{\k} A = A\otimes_{\k} A$, the statement of this theorem contains a non-trivial new construction. This is because, in view of the fact that we make no finiteness assumptions on $\k, A$,
the ordinary adic completion operation is ill-behaved from a homological perspective.
Thus, even in the case where $A\otimes^{\mrm{L}}_{\k} A$ is an ordinary commutative ring, 
we do not perform on it the usual adic completion operation, but instead its non-abelian derived functor. In particular, the derived adic completion of a non-noetherian commutative ring need not be a ring. Instead, it is a derived ring which may have higher cohomologies. Theorem \ref{thm:intro} suggests that this non abelian derived functor might have better formal properties than the usual adic completion operation. This might suggest that the fact that many results in the commutative algebra of adic rings were only proven under a noetherian assumption is due to the fact that in non-noetherian contexts, the derived completion operation is the correct operation needed to be considered.

In Section 1 below we gather some preliminaries from homological and homotopical algebra.
In Section 2 we extended the theory of local cohomology and derived completion to the derived category of DG-modules over a commutative DG-ring. In addition to being an interesting theory on its own, this theory is crucially used in the remaining of this paper. Section 3 introduces a non-abelian functor called adic completion on the category of pairs $(A,\bar{\a})$, where $A$ is a commutative DG-ring, and $\bar{\a} \subseteq \mrm{H}^0(A)$ is a finitely generated ideal. In Section 4 we prove that this non-abelian functor has a non-abelian derived functor (in the sense of homotopy theory), called the derived adic completion functor. We study its basic properties, and show it coincide with the usual adic completion operation on noetherian rings and on maps between noetherian rings.
The short Section 5 studies some interactions between the constructions of Section 2 and Section 4. Finally, in Section 6 we prove Theorem \ref{thm:intro} and discuss further applications of the theory developed in this paper.

\subsection*{}

We end this introduction by discussing some related works.
In \cite[Section 9]{IM} the authors describes a theory of local cohomology over a commutative,
not necessarily non-positive DG-ring. 
Under a noetherian assumption which holds in the applications they have for that theory in their paper,
their local cohomology coincides with ours.
The recent paper \cite{BHV} describes a general framework to describe isomorphisms generalizing the Greenlees-May duality.
Presumably, our version of it for commutative DG-rings (Proposition \ref{prop:GM}) falls under this framework.
Other related papers in this context are \cite{DG,DGI}, 
which describes a similar formal theory,
which, presumably could also be used to construct the theory of Section 2,
and the recent \cite{Po} which further extends the MGM equivalence.
Predecessors for the results of this section include \cite{GM} which was the first serious investigation of derived completion in an abelian setting,
and \cite{AJL1,AJL2} which globalized these results and extended them to the framework of derived categories on schemes and formal schemes.

Regarding Section 4, the only similar construction we are aware of is \cite[Section 4]{Lu},
where Lurie constructs derived completion in a spectral setting. 
Lurie's construction is easy to carry out,
but it seems difficult to work with it in commutative algebra,
nor any functorial properties of this construction are discussed there.
In contrast, our construction is quite difficult to carry out (it takes the majority of Section 4), 
but it is very easy to work with in commutative algebra:
we work in a strictly (graded) commutative setting, 
and to compute our derived completion, we identify a large class of commutative DG-rings (called weakly proregular DG-rings)
which serve as resolutions in this context,
on which derived completion is naturally isomorphic to ordinary adic completion.
Moreover, the fact that our construction is functorial is in some sense the main point of this paper, 
as it shows that there are interesting natural morphisms between derived completions of DG-rings and ordinary adically complete noetherian rings (see Section 6).

Finally, regarding the results of Sections 5 and 6,
as far as we know they are new, and never been carried out before in any similar non-abelian derived framework.

\section{Preliminaries}\label{sec:prel}

We will freely use the language of derived categories over rings and DG-rings,
including resolutions of unbounded complexes (\cite{Sp}) and unbounded DG-modules (\cite{Ke}).

\subsection{The category $\cdg_{\k}$ and its homotopy category}

We begin by recalling some basic notions from the theory of commutative DG-rings. 
We follow the definitions and notations of \cite{Ye2,Ye3}
(this includes using the term "DG-ring" instead of the somewhat more common "DG-algebra", 
see \cite[Remark 1.4]{Ye3} for a justification).
A commutative DG-ring is a $\mathbb{Z}$-graded ring
\[
A = \bigoplus_{n=-\infty}^{\infty} A^n
\]
and a $\mathbb{Z}$-linear differential $d:A \to A$ of degree $+1$,
such that $d(a\cdot b) = d(a)\cdot b + (-1)^i\cdot a \cdot d(b)$,
for all $a \in A^i$, $b \in A^j$,
and such that $b \cdot a = (-1)^{i\cdot j}\cdot a \cdot b$,
and $a \cdot a = 0$ if $i$ is an odd integer.
We shall further assume that $A$ is non-positive,
i.e, that $A^i = 0$ for all $i>0$.
For a DG-ring $A$, following \cite[Definition 1.3(2)]{Ye2}, 
we will sometimes denote $\mrm{H}^0(A)$ by $\bar{A}$.
Accordingly, ideals of the commutative ring $\bar{A}$ will be denoted by $\bar{\a}$.
The natural map $A \to \bar{A}$ will be denoted by $\pi_A$.
We denote by $\opn{DGMod}(A)$ the category of (unbounded) DG-modules over $A$,
and by $\cat{D}(A)$ its derived category. 
In particular, if $A = A^0$ is a ring, 
$\opn{DGMod}(A)$ is simply the category of unbounded complexes of $A$-modules.

Fixing a commutative DG-ring $\k$,
we denote by $\cdg_{\k}$ be the category of commutative non-positive DG-rings over $\k$.
A map in $\cdg_{\k}$ is a morphism of commutative DG-rings which is $\k$-linear.
If $\k = \mathbb{Z}$, the initial object of the category of commutative DG-rings,
we simply denote it by $\cdg$.
A map $A \to B$ of commutative DG-rings is called semi-free if it induces an isomorphism 
$B \cong A\otimes_{\mathbb{Z}} \mathbb{Z}[X]$ of graded commutative rings,
where $X$ is a graded set of variables. 
In this case, $B$ is K-projective over $A$. 
According to \cite[Theorem 3.21]{Ye3}, any map of commutative DG-rings $A \to B$
can be factored as $A \to \til{B} \to B$, such that $A \to \til{B}$ is semi-free, and $\til{B} \to B$ is a surjective quasi-isomorphism.

If $\k$ is a ring, and moreover it contains $\mathbb{Q}$, there is a model structure on $\cdg_{\k}$, 
in which the weak equivalences are the quasi-isomorphisms, the fibrations are the surjections, 
and the generating cofibrations are the semi-free maps. 
It is very convenient to do homotopy theory with this model structure. 
Unfortunately, if $\k$ does not contain $\mathbb{Q}$, 
there is no such model structure on $\cdg_{\k}$.
In general, according to \cite[Theorem 9.7]{St}, 
for any commutative ring $\k$, the category
$\cdg_{\k}$ has another model structure, with the expected weak equivalences and cofibrations.
However, it is not known what are the fibrations in this model structure.
In any case, since such a model structure exist,
we deduce immediately that the category $\cdg_{\k}$,
with weak equivalences being quasi-isomorphisms is a homotopical category,
in the sense of \cite[Section 33]{DHKS}. 
In particular, we may form its homotopy category $\ho(\cdg_{\k})$,
obtained by formally inverting quasi-isomorphisms, 
and deduce that it is locally small.

Despite the fact that quasi-isomorphisms, surjections and semi-free maps do not define a model structure on $\cdg_{\k}$ in general,
we shall now show that many features of model category theory remain true for these three classes of maps.
We begin by borrowing from Quillen's model category theory the notion of a left homotopy.

\begin{dfn}\label{def:left-homotopy}
Let $\k$ be a commutative DG-ring,
and let $f,g : A \to B$ be two maps in $\cdg_{\k}$. We say that $f,g$ are left homotopic over $\k$ if:
\begin{enumerate}
\item There is factorization
$A\otimes_{\k} A \xrightarrow{\alpha} C \xrightarrow{\beta} A$ in $\cdg_{\k}$ of the multiplication map $A\otimes_{\k} A \to A$, such that the map $\beta: C \to A$ is a quasi-isomorphism, and
\item Letting $\sigma_1,\sigma_2:A \to A\otimes_{\k} A$ be the two natural maps $A \to A\otimes_{\k} A$,
there is a map $h:C \to B$ in $\cdg_{\k}$, such that $h \circ \alpha \circ \sigma_1 = f$ and  $h \circ \alpha \circ \sigma_2 =g$.
\end{enumerate}
\end{dfn}

The following implication is standard in model category theory.

\begin{prop}\label{prop:homot-become-equal}
Let $\k$ be a commutative DG-ring,
denote by $Q_{\k}$ the localization functor $\cdg_{\k} \to \ho(\cdg_{\k})$,
and let $f,g:A \to B$ be left homotopic maps over $\k$.
Then 
\[
Q_{\k}(f) = Q_{\k}(g).
\]
\end{prop}
\begin{proof}
Following the notation of Definition \ref{def:left-homotopy},
we have that 
\[
Q_{\k}(1_A) = Q_{\k}(\beta \circ \alpha \circ \sigma_1) = Q_{\k}(\beta \circ \alpha \circ \sigma_2)
\]
But $Q_{\k}(\beta)$ is invertible, so we deduce that 
\[
Q_{\k}(\alpha \circ \sigma_1) = Q_{\k}(\alpha \circ \sigma_2),
\]
which implies that
\[
Q_{\k}(f) =  Q_{\k}(h \circ \alpha \circ \sigma_1) = Q_{\k}(h \circ \alpha \circ \sigma_2) = Q_{\k}(g).
\]
\end{proof}

The next result is an analogue of \cite[Proposition 7.6.13]{Hi} in this context.
\begin{prop}\label{prop:lift-homot}
Let $\k$ be a commutative DG-ring, 
and let
\[
\xymatrixcolsep{4pc}
\xymatrix{
A \ar[r]\ar[d]_{f} & B\ar[d]^{g}\\
C \ar[r] & D
}
\]
be a commutative diagram in $\cdg_{\k}$, such that $f$ is semi-free and $g$ is a surjective quasi-isomorphism.
\begin{enumerate}
\item There is a map $h_1:C \to B$ in $\cdg_{\k}$ making the diagram
\begin{equation}\label{eqn:hiscom}
\xymatrixcolsep{4pc}
\xymatrix{
A \ar[r]\ar[d]_{f} & B\ar[d]^{g}\\
C \ar[r]\ar@{.>}[ru]_{h_1} & D
}
\end{equation}
commutative.
\item If $h_2:C \to B$ is another map in $\cdg_{\k}$ making the diagram (\ref{eqn:hiscom}) commutative, then
$h_1,h_2$ are left homotopic over $A$.
\end{enumerate}
\end{prop}
\begin{proof}
The existence (1) is \cite[Theorem 3.22(i)]{Ye3}.
Assume $h_1,h_2:C \to B$ are two such maps.
Let $C\otimes_A C \xrightarrow{i} E \xrightarrow{j} C$ be a factorization of the multiplication map $C\otimes_A C \to C$ such that $i$ is semi-free and $j$ is a surjective quasi-isomorphism. Because $h_1 \circ f = h_2\circ f$, the two maps $h_1,h_2$ induce a map $h_1\otimes h_2: C\otimes_A C \to B$.
We obtain a commutative diagram
\[
\xymatrix{
C\otimes_A C \ar[d]_{i}\ar[rr]^{h_1\otimes h_2} & & B\ar[d]^{g}\\
E \ar[r]_{j} & C \ar[r] & D
}
\]
in $\cdg_A$, in which $i$ is semi-free, and $g$ is a surjective quasi-isomorphism.
Hence, by (1), there is a map $\phi: E \to B$, making the diagram
\[
\xymatrixcolsep{4pc}
\xymatrix{
C\otimes_A C \ar[d]_{i}\ar[rr]^{h_1\otimes h_2} & & B\ar[d]^{g}\\
E \ar[r]_{j} \ar[rru]^{\phi} & C \ar[r] & D
}
\]
commutative, so we get a diagram
\[
\xymatrixcolsep{4pc}
\xymatrix
{
C \ar@<-.5ex>[r] \ar@<.5ex>[r]  & C\otimes_{A} C \ar[r]^{i}  & E \ar[r]^{j}\ar[d]^{\phi} & C\\
            &                                     & B                   &       
}
\]
proving the claim.
\end{proof}
 
 We record some immediate corollaries of this proposition that will be used in the sequel.
 \begin{cor}\label{cor:lift-hom}
 Let $\k$ be a commutative DG-ring.
 \begin{enumerate}
\item Let $X \in \cdg_{\k}$, 
 and let $a:A \to X$ and $b:B\to X$ be two quasi-isomorphisms in $\cdg_{\k}$,
 such that $A,B$ are semi-free over $\k$,
 and such that $b$ is surjective.
 Then there is a $\k$-linear quasi-isomorphism $h:A \to B$, 
 unique up to a left homotopy over $\k$,
 making the diagram
 \[
 \xymatrix{
 A\ar[rd]_a \ar@{.>}[rr]^h & & B\ar[ld]^b\\
                               &X& 
 }
 \]
 commutative.
 \item Let $f:X \to Y$, $a:A \to X$ and $b:B \to Y$ be maps in $\cdg_{\k}$,
 and assume that $a,b$ are quasi-isomorphisms, that $A, B$ are semi-free over $\k$, and that $b$ is surjective.
 Then there is a map $\til{f}:A \to B$ in $\cdg_{\k}$, 
 unique up to a left homotopy over $\k$,
 making the diagram
\[
\xymatrix{
A \ar@{.>}[r]^{\til{f}}\ar[d]_a & B\ar[d]^b\\
X \ar[r]_f & Y
}
\]
commutative.
 \end{enumerate}
 \end{cor}
 \begin{proof}
These follow from applying the proposition to the commutative diagrams
 \[
 \xymatrix{
 \k \ar[r]\ar[d] & B\ar[d]^b &  \k \ar[r]\ar[d] & B\ar[d]^b\\
 A \ar[r]_a & X &  A \ar[r]_{f \circ a} & Y
 }
 \]
 \end{proof}

\subsection{Standard functors over $\ho(\cdg)$}

For the next constructions, we first make some set theoretic remarks.
Fix a Grothendieck universe $\mathcal{V}$,
and assume that $\cdg \in \mathcal{V}$,
and that for any $A \in \cdg$,
we have that
$\opn{DGMod}(A), \cat{D}(A) \in \mathcal{V}$.
Let $\cat{Cat}$ be the category of $\mathcal{V}$-small categories,
and fix a Grothendieck universe $\mathcal{U}$,
such that $\mathcal{V} \subseteq \mathcal{U}$,
and such that $\cat{Cat} \in \mathcal{U}$.
We give $\cat{Cat}$ the canonical model structure (see \cite{Re}).
Recall that the weak equivalences in this structure are exactly the equivalences of categories.
The homotopy category $\ho(\cat{Cat})$ has the same objects,
but its morphisms are natural equivalence classes of morphisms.
Any map $f:A \to B$ of commutative DG-rings give rise to these functors:
\[
- \otimes^{\mrm{L}}_A B, \quad \mrm{R}\opn{Hom}_A(B,-) : \cat{D}(A) \to \cat{D}(B)
\]
and the forgetful functor $\cat{D}(B) \to \cat{D}(A)$.
As is well known, 
if $f$ is a quasi-isomorphism then each of these three functors is an equivalence of categories.
We now use this fact to extend them to $\ho(\cdg)$.

Define functors 
\begin{equation}\label{eqn:Rf}
\mrm{L}(-)^* : \cdg \to \ho(\cat{Cat}), \quad \mrm{R}(-)_* : \cdg^{\opn{op}} \to \ho(\cat{Cat}), \quad \mrm{R}(-)^{\flat} : \cdg \to \ho(\cat{Cat}).
\end{equation}
as follows:
on objects, each of $\mrm{L}(-)^{*}$, $\mrm{R}(-)_{*}$ and $\mrm{R}(-)^{\flat}$ map $A \in \cdg$ to the category $\cat{D}(A) \in \ho(\cat{Cat})$.
Given $f:A \to B$ in $\cdg$, $\mrm{L}f^* : \cat{D}(A) \to \cat{D}(B)$ is given by $-\otimes^{\mrm{L}}_A B$, $\mrm{R}f_*:\cat{D}(B) \to \cat{D}(A)$ is the forgetful functor,
and $\mrm{R}(-)^{\flat}: \cat{D}(A) \to \cat{D}(B)$ is given by $\mrm{R}\opn{Hom}_A(B,-)$.
Since $\ho(\cat{Cat})$ identifies naturally equivalent functors, 
it is clear that $\mrm{L}(-)^*$, $\mrm{R}(-)_*$ and $\mrm{R}(-)^{\flat}$ are functors.
Moreover, they all map quasi-isomorphisms to isomorphisms,
so they induce functors
\[
\mrm{L}(-)^* : \ho(\cdg) \to \ho(\cat{Cat}),
\]
\[
\mrm{R}(-)_* : \ho(\cdg)^{\opn{op}} \to \ho(\cat{Cat}).
\]
and
\[
\mrm{R}(-)^{\flat} : \ho(\cdg) \to \ho(\cat{Cat}).
\]
In particular, using either of these functors,
we see that the operation 
\[
\ho(\cdg) \to \ho(\cat{Cat}), \quad A \mapsto \cat{D}(A)
\]
is well defined,
so that we may speak about the derived category of an object of $\ho(\cdg)$.

\subsection{Derived tensor products}\label{sec:dertenprod}

Fix a commutative DG-ring $\k$.
There is a functor 
\[
\cdg_{\k} \times \cdg_{\k} \to \cdg_{\k},
\]
given by $(A,B) \mapsto A\otimes_{\k} B$.
In this section we will construct its non-abelian derived functor.
We begin by constructing an analogue of the fibrant cofibrant replacement functor in this context.

\begin{prop}\label{prop:SF}
Let $\k$ be a commutative DG-ring,
and let $Q_{\k}: \cdg_{\k} \to \ho(\cdg_{\k})$ be the localization functor.
There is a functor $\opn{SF}_{\k}: \cdg_{\k} \to \ho(\cdg_{\k})$,
and a natural transformation $\opn{SF}_{\k} \to Q_{\k}$,
such that for any $A \in \cdg_{\k}$, $\opn{SF}_{\k}(A)$ is semi-free over $\k$, 
and the map $\opn{SF}_{\k}(A) \to A$ is equal to $Q_{\k}(i_A)$, where $i_A : \opn{SF}_{\k}(A) \to A$ is a surjective quasi-isomorphism.
\end{prop}
\begin{proof}
For any $A \in \cdg_{\k}$, factor $\k \to A$ as $\k \to \til{A} \xrightarrow{i_A} A$, 
with $\k \to \til{A}$ being semi-free, and $\til{A} \to A$ being a surjective quasi-isomorphism.
Given a map $f:A \to B$ in $\cdg_{\k}$, by Proposition \ref{prop:lift-homot}(1), 
there is a map $i_f:\til{A} \to \til{B}$ in $\cdg_{\k}$ such that $f \circ i_A = i_B \circ i_f$.
Let $\opn{SF}(A) := \til{A}$, and $\opn{SF}(f) := Q_{\k}(i_f)$.
If $g:B \to C$ is another map, then there are two commutative diagrams
\[
\xymatrixcolsep{5pc}
\xymatrixrowsep{5pc}
\xymatrix{
\k \ar[r]\ar[d] & \til{C}\ar[d] & \k \ar[r]\ar[d] & \til{C}\ar[d]\\
\til{A} \ar[ru]_{\opn{SF}(g\circ f)} \ar[r] & C & \til{A} \ar[ru]_{i_g \circ i_f} \ar[r] & C
}
\]
in which the left vertical map is semi-free and the right vertical map is a surjective quasi-isomorphism, 
so by Proposition \ref{prop:lift-homot}(2), we see that $\opn{SF}(g \circ f) = \opn{SF}(g)  \circ \opn{SF}(f)$,
which implies that indeed $\opn{SF}$ is a functor, and the collection of maps $i_A$ is a natural transformation.
\end{proof}

Using this functor, we may define the derived tensor product functor
\[
- \otimes^{\mrm{L}}_{\k} - : \ho(\cdg_{\k} \times \cdg_{\k}) = \ho(\cdg_{\k}) \times \ho(\cdg_{\k}) \to \ho(\cdg_{\k})
\]
by letting
\[
- \otimes^{\mrm{L}}_{\k} - := \opn{SF}_{\k}(-) \otimes_{\k} \opn{SF}_{\k}(-).
\]
This is well defined since $\opn{SF}_{\k}$ preserves quasi-isomorphisms,
and since for any $A\in \cdg_{\k}$, we have that $\opn{SF}_{\k}(A)$ is K-flat over $\k$.
K-flatness also ensures that there is a derived tensor product functor
\[
\cat{D}(A) \times \cat{D}(B) \to \cat{D}(A\otimes^{\mrm{L}}_{\k} B).
\]
In the special case where $A = B$, we can say more:
We may consider the functor 
\[
\Delta: \cdg_{\k} \to \cdg_{\k}, \quad A \mapsto A\otimes_{\k} A.
\]
This functor comes equipped with a natural transformation $\Delta \to 1_{\cdg_{\k}}$, 
given by the multiplication map $A\otimes_{\k} A \to A$.
Again, we may form the derived functor 
\[
\mrm{L}\Delta: \ho(\cdg_{\k}) \to \ho(\cdg_{\k}),
\]
and deduce that there is a natural transformation $\mrm{L}\Delta \to 1_{\ho(\cdg_{\k})}$.
Hence, for any $A \in \cdg_{\k}$, 
there is a natural map 
\begin{equation}\label{eqn:diag-map}
\Delta_A: A\otimes^{\mrm{L}}_{\k} A \to A
\end{equation}
in $\ho(\cdg_{\k})$,
given by the composition 
\[
A\otimes^{\mrm{L}}_{\k} A = \opn{SF}_{\k}(A) \otimes_{\k} \opn{SF}_{\k}(A) \to \opn{SF}_{\k}(A) \to A.
\]
The functor $(\Delta_A)^{\flat} : \cat{D}(A\otimes^{\mrm{L}}_{\k} A) \to \cat{D}(A)$ is called Shukla cohomology, 
or derived Hochschild cohomology.
See \cite{AILN,Shu,Ye3} for a study of this functor.

\subsection{Completion and torsion over commutative rings, and weakly proregular ideals}

Given a commutative ring $A$,
and a finitely generated ideal $\a\subseteq A$,
the $\a$-torsion and $\a$-adic completion functors are given by
\[
\Gamma_{\a}(-) := \varinjlim_n \opn{Hom}_A(A/\a^n,-), \quad \Lambda_{\a}(-) := \varprojlim_n A/\a^n \otimes_A -.
\]
These are both additive functors
\[
\opn{Mod}(A) \to \opn{Mod}(A),
\]
so they have derived functors
\[
\mrm{R}\Gamma_{\a}, \mrm{L}\Lambda_{\a}: \cat{D}(A) \to \cat{D}(A).
\]
It is clear that the functor $\Gamma_{\a}$ is idempotent. 
By \cite[Corollary 3.6]{Ye1}, $\Lambda_{\a}$ is also idempotent.
We denote by $\sqrt{\a}$ the radical of the ideal $\a$.
It is also an ideal, but it need not be finitely generated in general.
There are equalities
\begin{equation}\label{eqn:equal-of-sqrt}
\Gamma_{\a}(-) = \Gamma_{\sqrt{\a}}(-), \quad  \Lambda_{\a}(-) = \Lambda_{\sqrt{\a}}(-)
\end{equation}
of functors 
\[
\opn{Mod}(A) \to \opn{Mod}(A).
\]
Given an element $a \in A$,
the infinite dual Koszul complex associated to it, 
$\opn{K}^{\vee}_{\infty}(A;a)$ is the complex
\[
0 \to A \to A[a^{-1}] \to 0
\]
concentrated in degrees $0,1$.
Given a finite sequence $\mathbf{a} = (a_1,\dots,a_n)$ in $A$,
we let 
\[
\opn{K}^{\vee}_{\infty}(A;\mathbf{a}) := \opn{K}^{\vee}_{\infty}(A;a_1) \otimes_A \dots \otimes_A \opn{K}^{\vee}_{\infty}(A;a_n).
\]
This is a bounded complex of flat $A$-modules,
so in particular it is K-flat.
By \cite{AJL1, PSY1,Sc}, 
the sequence $\mathbf{a}$ is called \textbf{weakly proregular} if for any injective $A$-module $I$ and any $k>0$,
we have that
\[
H^k( \opn{K}^{\vee}_{\infty}(A;\mathbf{a})\otimes_A I) = 0.
\]
As shown in \cite[Proposition 1.2]{Sc},
this notion depends only on the ideal generated by $\mathbf{a}$.
Hence, we say that a finitely generated ideal is weakly proregular if some (equivalently, any) 
finite sequence that generates it is weakly proregular.
It follows (\cite[Lemma 3.1.1]{AJL1}, \cite[Theorem 1.1]{Sc}, \cite[Corollary 4.26]{PSY1})
that $\mathbf{a}$ and $\a$ are weakly proregular if and only if
there is an isomorphism
\[
\mrm{R}\Gamma_{\a}(-) \cong -\otimes_A \opn{K}^{\vee}_{\infty}(A;\mathbf{a})
\]
of functors $\cat{D}(A) \to \cat{D}(A)$.
According to \cite[Section 5]{PSY1},
the complex $\opn{K}^{\vee}_{\infty}(A;\mathbf{a})$ has an explicit bounded free resolution $\opn{Tel}(A;\mathbf{a})$,
called the telescope complex. 
In particular, $\opn{Tel}(A;\mathbf{a})$ is K-projective.
By \cite[Theorem 1.1]{Sc}, \cite[Corollary 5.25]{PSY1}, 
$\mathbf{a}$ and $\a$ are weakly proregular if and only if
there is an isomorphism
\[
\mrm{L}\Lambda_{\a}(-) \cong\opn{Hom}_A(\opn{Tel}(A;\mathbf{a}),-)
\]
of functors $\cat{D}(A) \to \cat{D}(A)$.
The telescope compelx satisfies the following base change property:
given a ring homomorphism $f:A \to B$,
letting $\mathbf{b}:= f(\mathbf{a})$,
there is an isomorphism
\begin{equation}\label{eqn:TelBaseChange}
\opn{Tel}(A;\mathbf{a}) \otimes_A B \cong \opn{Tel}(B;\mathbf{b})
\end{equation}
of complexes of $B$-modules.

The $A$-module $\Lambda_{\a}(A)$ has the structure of a commutative ring,
sometimes denoted by  $\widehat{A}$.
The map $A \to \widehat{A}$ is a map of commutative rings,
and if $A$ is noetherian, it is flat.

If $A$ is non-noetherian, this map can fail to be flat (even if $\a$ is weakly proregular).
For any $M \in \opn{Mod}(A)$,
the $A$-modules $\Gamma_{\a}(M), \Lambda_{\a}(M)$ have the structure of $\widehat{A}$-modules,
so one obtains additive functors
\[
\widehat{\Gamma}_{\a}, \widehat{\Lambda}_{\a}: \opn{Mod}(A) \to \opn{Mod}(\widehat{A}),
\]
and hence derived functors
\[
\mrm{R}\widehat{\Gamma}_{\a}, \mrm{L}\widehat{\Lambda}_{\a}: \cat{D}(A) \to \cat{D}(\widehat{A}).
\]
If $Q : \cat{D}(\widehat{A}) \to \cat{D}(A)$ is the forgetful functor,
there are natural isomorphisms
\[
Q \circ \mrm{R}\widehat{\Gamma}_{\a}\cong \mrm{R}\Gamma_{\a} ,\quad Q \circ \mrm{L}\widehat{\Lambda}_{\a} \cong \mrm{L}\Lambda_{\a}.
\]
According to \cite[Theorems 3.2, 3.6]{Sh1}, 
if $\a$ and $\mathbf{a}$ are weakly proregular then there are isomorphisms
\[
\mrm{R}\widehat{\Gamma}_{\a}(M) \cong (M\otimes_A \opn{K}^{\vee}_{\infty}(A;\mathbf{a})) \otimes^{\mrm{L}}_A \widehat{A}
\]
and
\[
\mrm{L}\widehat{\Lambda}_{\a}(M) \cong\mrm{R}\opn{Hom}_A(\opn{Tel}(A;\mathbf{a})\otimes_A \widehat{A},M)
\]
of functors $\cat{D}(A) \to \cat{D}(\widehat{A})$.
See \cite{Sh0,Sh1,Sh3,SW} for more information about these functors.

Let $A,B$ be commutative rings, 
and let $\a\subseteq A$, $\b\subseteq B$ be finitely generated ideals.
Given a map $f:A \to B$,
if $f(\a) \subseteq \b$ then $f$ is continuous with respect to the adic topologies generated by these ideals.
Of particular importance are the continuous maps $f$ such that $f(\a) \cdot B = \b$.
Such a map $f$ is called an adic map. 

\subsection{The category $\cdgcont$ and its homotopy category}\label{sec:cdgadcont}

Fix a commutative ring $\k$.
By an ideal of definition for an adic topology on a commutative DG-ring $A$,
we shall simply mean a finitely generated ideal 
$\bar{\a} \subseteq \bar{A} = \mrm{H}^0(A)$.

We define the category $\cdgcont[/\k]$ as follows: 
objects of this category are pairs $(A,\bar{\a})$,
where $A$ is a commutative DG-ring over $\k$,
and $\bar{\a}$ is an ideal of definition for an adic topology on $A$.
A morphism $f:(A,\bar{\a}) \to (B,\bar{\b})$ in $\cdgcont[/\k]$ is a $\k$-linear morphism of commutative DG-rings 
$f:A \to B$, such that $\mrm{H}^0(f)(\bar{\a}) \subseteq \bar{\b}$.

We do not know if $\cdgcont[/\k]$ has a natural model structure.
However, it does carry the structure of a homotopical category, 
in the sense of \cite[Section 33]{DHKS}. 
Weak equivalences in $\cdgcont[/\k]$ are by definition morphisms
$f:(A,\bar{\a}) \to (B,\bar{\b})$,
such that $f:A \to B$ is a quasi-isomorphism,
and $\mrm{H}^0(f)$ maps $\bar{\a}$ bijectively into $\bar{\b}$.
In other words, a weak equivalence is a quasi-isomorphism whose $\mrm{H}^0$ is an adic map.
Since quasi-isomorphisms are the weak equivalences of the homotopical category $\cdg_{\k}$,
it is clear that weak equivalences in $\cdgcont[/\k]$ satisfy the two out of three property and the two out of six property 
(in the sense of \cite[Section 33.1]{DHKS}).
Hence, $\cdgcont[/\k]$ is a homotopical category.

Every homotopical category has a homotopy category (\cite[Section 33.8]{DHKS}).
We denote the homotopy category of $\cdgcont[/\k]$ by $\ho(\cdgcont[/\k])$.
It is obtained from $\cdgcont[/\k]$ by formally inverting the weak equivalences.
As the forgetful functor 
\[
\cdgcont[/\k] \to \cdg_{\k}, \quad (A,\bar{\a})  \mapsto A
\]
clearly preserves weak equivalences,
it induces a forgetful functor
\[
\ho(\cdgcont[/\k]) \to \ho(\cdg_{\k}). 
\]
We will frequently apply this forgetful functor implicitly.

Note that because of the definition of weak equivalences in $\cdgcont[/\k]$,
the functor $\opn{SF}_{\k}$ and the natural transformation $\opn{SF}_{\k} \to Q_{\k}$ from Proposition \ref{prop:SF} may be lifted to a functor
\begin{equation}\label{eqn:liftedSF}
\opn{SF}_{\k}:\cdgcont[/\k] \to \ho(\cdgcont[/\k])
\end{equation}
and a natural transformation $\opn{SF}_{\k} \to Q_{\k}$,
where now $Q_{\k}$ is the localization functor of $\cdgcont[/\k]$.
The lifted functor is given by $\opn{SF}_{\k}(A,\bar{\a}) = (\opn{SF}_{\k}(A),\bar{\a}')$, where $\bar{\a}'$ is the preimage of $\bar{\a}$ under the isomorphism $\mrm{H}^0(\opn{SF}_{\k}(A)) \to \mrm{H}^0(A)$.

\section{Derived completion and derived torsion of DG-modules}\label{sec:RL}

Let $A$ be a commutative DG-ring, and let $\bar{\a}\subseteq \bar{A} = \mrm{H}^0(A)$ be a finitely generated ideal.
In this section we will associate to $(A,\bar{\a})$ a pair of functors
\[
\RG{\a}, \LL{\a}: \cat{D}(A) \to \cat{D}(A)
\]
which, in case $A$ is a noetherian ring (or a bit more generally, $A$ is a ring and $\bar{\a}$ is a weakly proregular ideal), will coincide with the ordinary 
derived $\bar{\a}$-torsion and derived $\bar{\a}$-adic completion functors.
This section is a chapter in ordinary homological algebra, 
and no homotopical methods will be used in it.

There is a short abstract approach to define these functors:
denote by $\cat{D}_{\bar{\a}-\opn{tor}}(A)$ the full subcategory of $\cat{D}(A)$ whose objects are DG $A$-modules $M$,
such that for all $n \in \mathbb{Z}$, the $\bar{A}$-module $\mrm{H}^n(M)$ is $\bar{\a}$-torsion. Since the subcategory of $\opn{Mod}(\mrm{H}^0(A))$ consisting of $\bar{\a}$-torsion modules is a thick abelian subcategory of $\opn{Mod}(\mrm{H}^0(A))$,
it follows that $\cat{D}_{\bar{\a}-\opn{tor}}(A)$ is a triangulated subcategory of $\cat{D}(A)$. Elements of this category are called cohomologically $\bar{\a}$-torsion DG-modules. By \cite[Theorem 8.4.4]{Ne}, the inclusion functor
\[
Q: \cat{D}_{\bar{\a}-\opn{tor}}(A) \to \cat{D}(A)
\]
has a right adjoint 
\[
G: \cat{D}(A) \to \cat{D}_{\bar{\a}-\opn{tor}}(A),
\]
which is also triangulated. One may set $\RG{\a} := Q \circ G$.
Then, similarly, the functor $\RG{\a}$ has a left adjoint, 
denoted by $\LL{\a}$. Below, we will take a different approach define these functors,
which will allow us to explicitly compute them. The approach is based on the following, the key technical definition of this paper:

\begin{dfn}\label{def:wpr}
\leavevmode
\begin{enumerate}
\item A weakly proregular DG-ring is a pair $(A,\mathbf{a})$,
such that $A$ is a DG-ring which is K-flat over $A^0$,
and $\mathbf{a}$ is a finite sequence of elements in $A^0$ which is weakly proregular.
\item Given a commutative DG-ring $A$ and a finitely generated ideal $\bar{\a}\subseteq \mrm{H}^0(A)$,
a weakly proregular resolution of $(A,\bar{\a})$ is a triple $(f,B,\mathbf{b})$,
such that $(B,\mathbf{b})$ is a weakly proregular DG-ring,
$f:B \to A$ is a quasi-isomorphism of DG-rings,
and such that
\[
\pi_A(f(\mathbf{b}))\cdot \mrm{H}^0(A) = \bar{\a}.
\]
\item Given a weakly proregular DG-ring $(A,\mathbf{a})$,
we define its completion $\widehat{A}$ to be the $(\mathbf{a})$-adic completion of $A$.
\end{enumerate}
\end{dfn}

\begin{prop}\label{prop:res-exists}
Let $A$ be a commutative DG-ring, and let $\bar{\a}\subseteq \bar{A}$ be a finitely generated ideal.
Then $(A,\bar{\a})$ has a weakly proregular resolution $(f,B,\mathbf{b})$.
Moreover, for any finite sequence of elements $\mathbf{a} = (a_1,\dots,a_n)$ in $A^0$,
whose image in $\bar{A}$ generates $\bar{\a}$, the pair $(A,\bar{\a})$ has a weakly proregular resolution $(f,B,\mathbf{b})$ such that $f(\mathbf{b}) = \mathbf{a}$.
\end{prop}
\begin{proof}
Let $\mathbf{a} = (a_1,\dots,a_n)$ be a finite sequence of elements in $A^0$
such that $\mathbf{a} \cdot \bar{A} = \bar{\a}$.
Let $\mathbb{Z}[x_1,\dots,x_n] \to A$ be the map $x_i \mapsto a_i$,
and let 
\[
\mathbb{Z}[x_1,\dots,x_n] \to B \xrightarrow{f} A
\]
be a semi-free resolution of this map.
Letting $\mathbf{b}$ be the image of $(x_1,\dots,x_n)$ in $B^0$,
it is clear that $f(\mathbf{b}) = \mathbf{a}$,
and one may verify (see Lemma \ref{lem:sfr-is-wpr} below) that $(f,B,\mathbf{b})$ is a weakly proregular resolution of $(A,\bar{\a})$. 
\end{proof}

We will now associate to a weakly proregular resolution $P$ of $(A,\bar{\a})$ 
derived torsion and derived completion functors, 
which we will temporarily denote by
$\RG{\a}^P(-), \LL{\a}^P(-)$. 
This notation is temporary, 
as it will be shown below that these functors are independent of the chosen weakly proregular resolution, and are infact naturally isomorphic to the functors $\RG{\a}$ and $\LL{\a}$ introduced above.

\begin{dfn}\label{dfn:RGP}
Let $A$ be a commutative DG-ring, let $\bar{\a}\subseteq \bar{A} = \mrm{H}^0(A)$ be a finitely generated ideal,
and let $P = (f,B,\mathbf{b})$ be a weakly proregular resolution of $(A,\bar{\a})$.
Let $\b$ be the ideal in $B^0$ generated by $\mathbf{b}$.
We define the derived torsion and derived completion functors associated to $P$ as follows:
the operations 
\[
\Gamma_{\b}(-) := \varinjlim_n \opn{Hom}_{B^0}(B^0/{\b}^n,-), \quad \Lambda_{\b}(-) := \varprojlim_n B^0/{\b}^n \otimes_{B^0} -
\]
define additive functors $\opn{DGMod}(B) \to \opn{DGMod}(B)$.
Let
\[
\mrm{R}\Gamma_{\b}, \mrm{L}\Lambda_{\b} : \cat{D}(B) \to \cat{D}(B)
\]
be their derived functors, and set
\[
\RG{\a}^P(-) := \mrm{L}f^*(\mrm{R}\Gamma_{\b}(\mrm{R}f_*(-))) \quad \LL{\a}^P(-) := \mrm{R}f^{\flat}(\mrm{L}\Lambda_{\b}(\mrm{R}f_*(-))).
\]
These are triangulated functors $\cat{D}(A) \to \cat{D}(A)$.
\end{dfn}

The next proposition provides explicit formulas for computing $\RG{\a}^P, \LL{\a}^P$.
\begin{prop}\label{prop:form}
Let $A$ be a commutative DG-ring, let $\bar{\a}\subseteq \bar{A}$ be a finitely generated ideal,
and let $P = (f,B,\mathbf{b})$ be a weakly proregular resolution of $(A,\bar{\a})$.
Let $\b$ be the ideal in $B^0$ generated by $\mathbf{b}$.
Then there are isomorphisms
\[
\mrm{R}\Gamma_{\b}(-) \cong \opn{Tel}(B^0;\mathbf{b})\otimes_{B^0} -
\]
and
\[
\mrm{L}\Lambda_{\b}(-) \cong \opn{Hom}_{B^0}(\opn{Tel}(B^0;\mathbf{b}),-)
\]
of functors $\cat{D}(B) \to \cat{D}(B)$. 
Hence, letting $\mathbf{a} = f(\mathbf{b})$, there are isomorphisms
\[
\RG{\a}^P (-) \cong \opn{Tel}(A^0;\mathbf{a})\otimes_{A^0} - 
\]
and
\[
\LL{\a}^P(-) \cong \opn{Hom}_{A^0}(\opn{Tel}(A^0;\mathbf{a}),-)
\]
of functors $\cat{D}(A) \to \cat{D}(A)$.
\end{prop}
\begin{proof}
Let $M \in \cat{D}(B)$, let $M \iso I$ be a K-injective resolution over $B$,
and let $P \iso M$ be a K-projective resolution over $B$.
By definition, we have that
\[
\mrm{R}\Gamma_{\b}(M) = \Gamma_{\b}(I)
\]
and
\[
\mrm{L}\Lambda_{\b}(M) = \Lambda_{\b}(P).
\]
Let $\opn{Rest}_{B/B^0} : \cat{D}(B) \to \cat{D}(B^0)$ be the forgetful functor.
Because the map $B^0 \to B$ is K-flat, 
it follows that $\opn{Rest}_{B/B^0}(I)$ is K-injective over $B^0$,
and that $\opn{Rest}_{B/B^0}(P)$ is K-flat over $B^0$.

According to \cite[Equation (4.19)]{PSY1}, 
there is a functorial $B^0$-linear homomorphism
\[
v_{\mathbf{b},\opn{Rest}_{B/B^0}(I)} : \Gamma_{\b}(\opn{Rest}_{B/B^0}(I)) \to \opn{K}^{\vee}_{\infty}(B^0;\mathbf{b}) \otimes_{B^0} \opn{Rest}_{B/B^0}(I),
\]
and since $I$ has a $B$-linear structure, it is easy to check that this map is $B$-linear, that is,
there is a $B$-linear map
\begin{equation}\label{eqn:gammaToKoszul}
v_{\mathbf{b},I} : \Gamma_{\b}(I) \to \opn{K}^{\vee}_{\infty}(B^0;\mathbf{b}) \otimes_{B^0} I
\end{equation}
such that 
\[
\opn{Rest}_{B/B^0} ( v_{\mathbf{b},I} ) = v_{\mathbf{b},\opn{Rest}_{B/B^0}(I)}.
\]
Hence, to check that $v_{\mathbf{b},I}$ is a quasi-isomorphism, it is enough to show that
$v_{\mathbf{b},\opn{Rest}_{B/B^0}(I)}$ is a quasi-isomorphism, and this follows from \cite[Corollary 4.25]{PSY1}.
Thus, there are functorial isomorphisms
\[
\mrm{R}\Gamma_{\b}(M) \cong \opn{K}^{\vee}_{\infty}(B^0;\mathbf{b}) \otimes_{B^0} I \cong
\opn{K}^{\vee}_{\infty}(B^0;\mathbf{b}) \otimes_{B^0} M \cong \opn{Tel}(B^0;\mathbf{b})\otimes_{B^0} M,
\]
proving the first claim.
The second statement is proved similarly, using \cite[Definition 5.16]{PSY1} and \cite[Corollary 5.23]{PSY1}.
In particular, for any K-flat DG-module $P$, there is a functorial $B$-linear isomorphism
\begin{equation}\label{eqn:lambdaToTelescope}
\opn{tel}_{\mathbf{b},P} : \opn{Hom}_{B^0}(\opn{Tel}(B^0;\mathbf{b}),P) \to \Lambda_{\b}(P).
\end{equation}

The statements on $\RG{\a}^P$ and $\LL{\a}^P$ now follow from this and from the base change property of the telescope complex (\ref{eqn:TelBaseChange}).
\end{proof}

\begin{prop}\label{prop:LR}
Let $A$ be a commutative DG-ring, let $\bar{\a}\subseteq \bar{A}$ be a finitely generated ideal,
and let $P = (f,B,\mathbf{b})$ be a weakly proregular resolution of $(A,\bar{\a})$.
Then there is an isomorphism
\[
\LL{\a}^P(\RG{\a}^P(-)) \cong \LL{\a}^P(-)
\]
of functors $\cat{D}(A) \to \cat{D}(A)$.
\end{prop}
\begin{proof}
Let $M \in \cat{D}(A)$,
and let $\b$ be the ideal in $B^0$ generated by $\mathbf{b}$.
By definition, we have natural isomorphisms
\[
\begin{aligned}
\LL{\a}^P(\RG{\a}^P(M)) =\nonumber\\
\mrm{R}f^{\flat} \circ \mrm{L}\Lambda_{\b} \circ \mrm{R}f_* \circ \mrm{L}f^* \circ \mrm{R}\Gamma_{\b} \circ \mrm{R}f_*(M) \cong\nonumber\\
\mrm{R}f^{\flat} \circ \mrm{L}\Lambda_{\b} \circ \mrm{R}\Gamma_{\b} \circ \mrm{R}f_*(M),\nonumber
\end{aligned}
\]
and using Proposition \ref{prop:form}, we have an isomorphism of functors
\[
\mrm{R}f^{\flat} \circ \mrm{L}\Lambda_{\b} \circ \mrm{R}\Gamma_{\b} \circ \mrm{R}f_*(M) \cong
\mrm{R}f^{\flat} \circ \opn{Hom}_{B^0}(\opn{Tel}(B^0;\mathbf{b}), \opn{Tel}(B^0;\mathbf{b})\otimes_{B^0} \mrm{R}f_*(M)).
\]
According to \cite[Equation (5.6)]{PSY1}, 
there is a $B^0$-linear homomorphism
\[
u_{\mathbf{b}} : \opn{Tel}(B^0;\mathbf{b}) \to B^0.
\]
Setting $T:= \opn{Tel}(B^0;\mathbf{b})$ and $N:=\mrm{R}f_*(M)$, this map give rise to a $B$-linear homomorphism
\[
\opn{Hom}(1_T,\opn{Hom}(u_{\mathbf{b}},1_N)) : \opn{Hom}_{B^0} ( T, N) \to \opn{Hom}_{B^0} (T, \opn{Hom}_{B^0} (T, N)).
\]
According to \cite[Lemma 7]{PSY1E}, $\opn{Hom}(1_T,\opn{Hom}(u_{\mathbf{b}},1_N))$ is a quasi-isomorphism,
so we deduce that there is a natural isomorphism $\mrm{L}\Lambda_{\b}(\mrm{R}\Gamma_{\b}(M)) \cong \mrm{L}\Lambda_{\b}(M)$ in $\cat{D}(B)$, 
and clearly this is enough to establish the claim.
\end{proof}

\begin{rem}
Let $A$ be a commutative DG-ring, let $\bar{\a}\subseteq \bar{A}$ be a finitely generated ideal, 
and let $P = (f,B,\mathbf{b})$ be a weakly proregular resolution of $(A,\bar{\a})$.
For any $M \in \cat{D}(A)$, there is a natural map 
\[
\sigma^P_M : \RG{\a}^P(M) \to M
\]
in $\cat{D}(A)$ defined as follows:
Let $N := \mrm{R}f_*(M) \in \cat{D}(B)$, and let $N \iso I$ be a K-injective resolution over $B$.
We have a sequence of natural morphisms:
\[
\mrm{R}\Gamma_{\b} ( N ) = \Gamma_{\b}(I) \to I \cong N
\]
in $\cat{D}(B)$. Let $\sigma^{\b}_N$ be the composition of these maps,
and define $\sigma^P_M$ to be the composition of $\mrm{L}f^*(\sigma^{\b}_{N})$ 
with the natural isomorphism $\mrm{L}f^*(N) \cong M$.
\end{rem}

\begin{prop}\label{prop:tor-properties}
Let $A$ be a commutative DG-ring, 
let $\bar{\a}\subseteq \bar{A}$ be a finitely generated ideal,
and let $P = (f,B,\mathbf{b})$
be a weakly proregular resolution of $(A,\bar{\a})$.
Then the following holds:
\begin{enumerate}
\item For any $M \in \cat{D}(A)$,
we have that
\[
\RG{\a}^P(M) \in \cat{D}_{\bar{\a}-\opn{tor}}(A).
\]
\item For any $N \in \cat{D}_{\bar{\a}-\opn{tor}}(A)$,
the map 
\[
\sigma^P_N : \RG{\a}^P(N) \to N
\]
is an isomorphism in $\cat{D}(A)$.
\item For any $M \in \cat{D}(A)$,
the map 
\[
\sigma^P_{\RG{\a}^P(M)} : \RG{\a}^P( \RG{\a}^P(M) ) \to \RG{\a}^P(M)
\]
is an isomorphism in $\cat{D}(A)$.
\end{enumerate}
\end{prop}
\begin{proof}
Let $\b = \mathbf{b} \cdot B^0$,
and let $\bar{\b}$ be the ideal in $\mrm{H}^0(B)$ generated by the image of $\b$.
It is clear that the equivalence of categories
\[
\mrm{R}f_* : \cat{D}(A) \to \cat{D}(B)
\]
restricts to an equivalence of categories
\[
\mrm{R}f_* : \cat{D}_{\bar{\a}-\opn{tor}}(A) \to \cat{D}_{\bar{\b}-\opn{tor}}(B)
\]
with a quasi-inverse
\[
\mrm{L}f^* : \cat{D}_{\bar{\b}-\opn{tor}}(B) \to \cat{D}_{\bar{\a}-\opn{tor}}(A).
\]
It follows that in order to prove (1), (2), we may assume without loss of generality that $A = B$.
Under this assumption, let us set $\a = \b$.
Let $M \in \cat{D}(A)$, and let $M \iso I$ be a K-injective resolution over $A$. 
Then 
\[
\RG{\a}^P(M) = \Gamma_{\a}(I).
\]
It follows that for all $n$, 
the $A^0$-module $\mrm{H}^n(\Gamma_{\a}(I))$ is $\a$-torsion.
But the $A^0$-action on $\mrm{H}^n(\Gamma_{\a}(I))$ factors through $\bar{A}$,
and $\a \cdot \bar{A} = \bar{\a}$, so we deduce that
\[
\RG{\a}^P(M) \in \cat{D}_{\bar{\a}-\opn{tor}}(A).
\]
This proves (1).
Now, let $N \in \cat{D}_{\bar{\a}-\opn{tor}}(A)$, 
and let $N \iso I$ be a K-injective resolution over $A$.
Note that 
\[
\opn{Rest}_{A/A^0}(I) \in \cat{D}_{\a-\opn{tor}}(A^0).
\]
It follows from \cite[Corollary 4.32]{PSY1} that the natural $A^0$-linear map
\[
\Gamma_{\a}(\opn{Rest}_{A/A^0}(I)) \inj \opn{Rest}_{A/A^0}(I)
\]
is a quasi-isomorphism, so the $A$-linear map $\Gamma_{\a}(I) \inj I$ is also a quasi-isomorphism,
and this implies that $\sigma^P_N$ is an isomorphism in $\cat{D}(A)$.
This establishes (2). 
Now (3) clearly follows from (1), (2).
\end{proof}

Dually to the above,
there is a natural map $\tau^P_M: M \to \mrm{L}\Lambda^P_{\bar{\a}}(M)$.
We denote by $\cat{D}(A)^P_{\opn{\bar{\a}-com}}$ the full subcategory of $\cat{D}(A)$ on which $\tau^P_M$ is an isomorphism.
The next result is proven similarly to Proposition \ref{prop:tor-properties}.

\begin{prop}\label{prop:com-properties}
Let $A$ be a commutative DG-ring, 
let $\bar{\a}\subseteq \bar{A}$ be a finitely generated ideal,
and let $P = (f,B,\mathbf{b})$
be a weakly proregular resolution of $(A,\bar{\a})$.
Then for any $M \in \cat{D}(A)$,
we have that 
\[
\LL{\a}^P(M)  \in \cat{D}(A)^P_{\opn{\bar{\a}-com}}.
\]
\end{prop}

The next result generalizes the Greenlees-May duality: 

\begin{prop}\label{prop:GM}
Let $A$ be a commutative DG-ring, let $\bar{\a}\subseteq \bar{A}$ be a finitely generated ideal,
and let $P = (f,B,\mathbf{b})$ be a weakly proregular resolution of $(A,\bar{\a})$.
Then for any $M, N \in \cat{D}(A)$,
there are bifunctorial isomorphisms
\[
\mrm{R}\opn{Hom}_A(\RG{\a}^P(M), N) \cong 
\mrm{R}\opn{Hom}_A(M, \LL{\a}^P(N)) \cong
\mrm{R}\opn{Hom}_A(\RG{\a}^P(M),\RG{\a}^P(N)).
\]
in $\cat{D}(A)$.
\end{prop}
\begin{proof}
Let $M' = \mrm{R}f_*(M)$,
and let $N' = \mrm{R}f_*(N)$,
and let $\b := \mathbf{b} \cdot B^0$.
By definition, 
we have that
\[
\mrm{R}\opn{Hom}_A(\RG{\a}^P(M), N) \cong 
\mrm{R}\opn{Hom}_A ( \mrm{L}f^*(\mrm{R}\Gamma_{\b}(M')), \mrm{L}f^*(N'))
\]
By \cite[Proposition 2.5]{Sh2},
there is a bifunctorial isomorphism
\[
\mrm{R}\opn{Hom}_A ( \mrm{L}f^*(\mrm{R}\Gamma_{\b}(M')), \mrm{L}f^*(N'))
\cong 
\mrm{L}f^* \mrm{R}\opn{Hom}_B( \mrm{R}\Gamma_{\b}(M') , N').
\]
Using Proposition \ref{prop:form},
we have that
\[
\mrm{R}\opn{Hom}_B( \mrm{R}\Gamma_{\b}(M') , N') \cong
\mrm{R}\opn{Hom}_B(\opn{Tel}(B^0;\mathbf{b})\otimes_{B^0} M',  N'),
\]
and using adjunction and Proposition \ref{prop:form} again,
we see that
\begin{eqnarray}
\mrm{R}\opn{Hom}_B(\opn{Tel}(B^0;\mathbf{b})\otimes_{B^0} M',  N') \cong\nonumber\\
\mrm{R}\opn{Hom}_B(M', \opn{Hom}_{B^0}(\opn{Tel}(B^0;\mathbf{b}), N')) \cong\nonumber\\
\mrm{R}\opn{Hom}_B(M', \mrm{L}\Lambda_{\b}(N')).\nonumber
\end{eqnarray}
Combining this chain of natural isomorphisms with an application of \cite[Proposition 2.5]{Sh2} again,
we see that
\[
\mrm{R}\opn{Hom}_A(\RG{\a}^P(M), N) \cong 
\mrm{R}\opn{Hom}_A(M, \LL{\a}^P(N)).
\]
For the second claim, 
note that by Proposition \ref{prop:LR},
we have that
\[
\mrm{R}\opn{Hom}_B(M', \mrm{L}\Lambda_{\b}(N')) \cong
\mrm{R}\opn{Hom}_B(M', \mrm{L}\Lambda_{\b}(\mrm{R}\Gamma_{\b}(N'))).
\]
Using Proposition \ref{prop:form} and adjunction,
we see that
\begin{eqnarray}
\mrm{R}\opn{Hom}_B(M', \mrm{L}\Lambda_{\b}(\mrm{R}\Gamma_{\b}(N'))) \cong\nonumber\\
\mrm{R}\opn{Hom}_B(M', \opn{Hom}_{B^0}(\opn{Tel}(B^0;\mathbf{b}),\mrm{R}\Gamma_{\b}(N'))) \cong\nonumber\\
\mrm{R}\opn{Hom}_B(M' \otimes_{B^0} \opn{Tel}(B^0;\mathbf{b}), \mrm{R}\Gamma_{\b}(N')) \cong\nonumber\\
\mrm{R}\opn{Hom}_B(\mrm{R}\Gamma_{\b}(M'), \mrm{R}\Gamma_{\b}(N')).\nonumber
\end{eqnarray}
so the second result also follows from an application of \cite[Proposition 2.5]{Sh2}.
\end{proof}

Using Proposition \ref{prop:tor-properties}(1), 
we see that $\RG{\a}^P$ defines a functor 
\[
\cat{D}(A) \to \cat{D}_{\bar{\a}-\opn{tor}}(A).
\]
Here is the main result of this section.

\begin{thm}\label{thm:uniq}
Let $A$ be a commutative DG-ring, 
and let $\bar{\a}\subseteq \bar{A}$ be a finitely generated ideal.
\begin{enumerate}
\item For any weakly proregular resolution $P = (f,B,\mathbf{b})$ of $(A,\bar{\a})$,
the functor
\[
\RG{\a}^P : \cat{D}(A) \to \cat{D}_{\bar{\a}-\opn{tor}}(A).
\]
is right adjoint to the inclusion functor $\cat{D}_{\bar{\a}-\opn{tor}}(A) \inj \cat{D}(A)$.
\item For any weakly proregular resolution $P = (f,B,\b)$ of $(A,\bar{\a})$,
the functor 
\[
\RG{\a}^P : \cat{D}(A) \to \cat{D}(A)
\]
is left adjoint to the functor
\[
\LL{\a}^P : \cat{D}(A) \to \cat{D}(A)
\]
\end{enumerate}
\end{thm}
\begin{proof}
Denote by $F(-)$ the inclusion functor  $\cat{D}_{\bar{\a}-\opn{tor}}(A) \inj \cat{D}(A)$.
Let $M \in \cat{D}_{\bar{\a}-\opn{tor}}(A)$,
and let $N \in \cat{D}(A)$.
By Proposition \ref{prop:tor-properties}(2),
we have a natural isomorphism
\[
\opn{Hom}_{\cat{D}(A)} ( F(M), N) \cong
\opn{Hom}_{\cat{D}(A)} ( \RG{\a}^P(M) , N) = 
\mrm{H}^0 \mrm{R}\opn{Hom}_A ( \RG{\a}^P(M) , N),
\]
and by Proposition \ref{prop:GM}, 
there is a natural isomorphism
\[
\mrm{H}^0 \mrm{R}\opn{Hom}_A ( \RG{\a}^P(M) , N) \cong
\mrm{H}^0 \mrm{R}\opn{Hom}_A ( \RG{\a}^P(M) , \RG{\a}^P(N)).
\]
Using Proposition \ref{prop:tor-properties}(2) again,
we see that there is a natural isomorphism
\begin{eqnarray}
\mrm{H}^0 \left(\mrm{R}\opn{Hom}_A ( \RG{\a}^P(M) , \RG{\a}^P(N)) \right) \cong\nonumber\\
\mrm{H}^0 \left( \mrm{R}\opn{Hom}_A ( M , \RG{\a}^P(N)) \right) =\nonumber\\
\opn{Hom}_{\cat{D}(A)} ( M , \RG{\a}^P(N)).\nonumber
\end{eqnarray}
Since $M, \RG{\a}^P(N) \in \cat{D}_{\bar{\a}-\opn{tor}}(A)$,
and since $\cat{D}_{\bar{\a}-\opn{tor}}(A)$ is a full subcategory of $\cat{D}(A)$, 
we have that
\[
\opn{Hom}_{\cat{D}(A)} ( M , \RG{\a}^P(N)) = 
\opn{Hom}_{\cat{D}_{\bar{\a}-\opn{tor}}(A)} ( M , \RG{\a}^P(N)).
\]
Thus, we see that there is a natural isomorphism
\[
\opn{Hom}_{\cat{D}(A)} ( F(M), N) \cong \opn{Hom}_{\cat{D}_{\bar{\a}-\opn{tor}}(A)} ( M , \RG{\a}^P(N)),
\]
and this proves (1). 
The second statement follows immediately from Proposition \ref{prop:GM}.
\end{proof}

\begin{cor}\label{cor:uniq}
Let $A$ be a commutative DG-ring, 
and let $\bar{\a}\subseteq \bar{A}$ be a finitely generated ideal.
Let $P,Q$ be weakly proregular resolutions of $(A,\bar{\a})$.
\begin{enumerate}
\item Let
\[
\varphi^P : \opn{Hom}_{\cat{D}(A)} ( M, N) \iso \opn{Hom}_{\cat{D}_{\bar{\a}-\opn{tor}}(A)} ( M , \RG{\a}^P(N))
\]
and
\[
\varphi^Q : \opn{Hom}_{\cat{D}(A)} ( M, N) \iso \opn{Hom}_{\cat{D}_{\bar{\a}-\opn{tor}}(A)} ( M , \RG{\a}^Q(N))
\]
be the adjunctions constructed in Theorem \ref{thm:uniq}(1). \\
Then there is \textbf{a unique isomorphism of functors} $\alpha_{P,Q} : \RG{\a}^P \iso \RG{\a}^Q$ such that 
\[
\varphi^Q = \opn{Hom}_{\cat{D}(A)} ( 1, \alpha_{P,Q} ) \circ \varphi^P.
\]
\item Let
\[
\psi^P : \opn{Hom}_{\cat{D}(A)} ( \RG{\a}^P(M), N) \iso \opn{Hom}_{\cat{D}(A)} ( M, \LL{\a}^P(N)) 
\]
and
\[
\psi^Q : \opn{Hom}_{\cat{D}(A)} ( \RG{\a}^Q(M), N) \iso \opn{Hom}_{\cat{D}(A)} ( M, \LL{\a}^Q(N)) 
\]
be be the adjunctions constructed in Theorem \ref{thm:uniq}(2). \\
Then there is \textbf{a unique isomorphism of functors} $\beta_{P,Q} : \LL{\a}^P \iso \LL{\a}^Q$ such that
\[
\psi^Q = \opn{Hom}_{\cat{D}(A)} (1,\beta_{P,Q}) \circ \psi^P \circ \opn{Hom}_{\cat{D}(A)} (\alpha_{P,Q}, 1).
\]
\end{enumerate}
\end{cor}
\begin{proof}
Both statements follow from Theorem \ref{thm:uniq} and the uniqueness of adjoint functors (\cite[Theorem 3.2]{Ka}).
\end{proof}

\begin{rem}
Let $A$ be a commutative DG-ring, 
and let $\bar{\a}\subseteq \bar{A}$ be a finitely generated ideal.
By Proposition \ref{prop:res-exists}, any such pair $(A,\bar{\a})$
has a weakly proregular resolution. 
We associate to $(A,\bar{\a})$ functors
\[
\RG{\a}, \LL{\a}: \cat{D}(A) \to \cat{D}(A)
\]
defined by choosing some weakly proregular resolution $P$ of $(A,\bar{\a})$,
and declaring 
\[
\RG{\a} := \RG{\a}^P, \quad \LL{\a} := \LL{\a}^P. 
\]
By Corollary \ref{cor:uniq}
up to a unique natural isomorphism respecting the structure,
these are independent of the chosen weakly proregular resolution, 
so this notation makes sense.
Similarly, we let 
\[
\cat{D}(A)_{\opn{\bar{\a}-com}} := \cat{D}(A)^P_{\opn{\bar{\a}-com}}.
\]
\end{rem}

It follows from the above that the functor $\cat{D}(A) \to \cat{D}_{\bar{\a}-\opn{tor}}(A)$
which is right adjoint to the inclusion functor is a (right) Bousfield localization. 
See \cite[Section 4.9]{Kr} for details about this notion in the context of triangulated categories.

\begin{cor}\label{cor:sqrt}
Let $A$ be a commutative DG-ring, 
and let $\bar{\a} , \bar{\b} \subseteq \bar{A}$ be finitely generated ideals 
such that $\sqrt{\bar{\a}} = \sqrt{\bar{\b}}$.
Then there are isomorphisms
\[
\mrm{R}\Gamma_{\bar{\a}}(-) \cong \mrm{R}\Gamma_{\bar{\b}}(-)
\]
and
\[
\mrm{L}\Lambda_{\bar{\a}}(-) \cong \mrm{L}\Lambda_{\bar{\b}}(-)
\]
of functors $\cat{D}(A) \to \cat{D}(A)$.
\end{cor}
\begin{proof}
The second claim follows from the first, 
and the first claim follows from the fact that 
\[
\cat{D}_{\bar{\a}-\opn{tor}}(A) = \cat{D}_{\sqrt{\bar{\a}}-\opn{tor}}(A) = \cat{D}_{\bar{\b}-\opn{tor}}(A) 
\]
since a $\mrm{H}^0(A)$-module is $\bar{\a}$-torsion if and only if it is $\sqrt{\bar{\a}}$-torsion.
\end{proof}

\begin{prop}\label{prop:LLOfHom}
Let $f:(A,\bar{\a}) \to (B,\bar{\b})$ be a map in $\cdgcont$,
such that 
\[
\mrm{H}^0(f)(\bar{\a})\cdot \mrm{H}^0(B) = \bar{\b}.
\]
Then there is an isomorphism
\[
\mrm{L}\Lambda_{\bar{\b}}(f^{\flat}(-)) \cong f^{\flat}(\mrm{L}\Lambda_{\bar{\a}}(-))
\]
of functors $\cat{D}(A) \to \cat{D}(B)$.
\end{prop}
\begin{proof}
Choose weakly proregular resolutions $P = (g,P_A,\mathbf{a})$ and $Q = (h,Q,\mathbf{b})$ 
of $(A,\bar{\a})$ and $(B,\bar{\b})$ respectively, 
such that there is a commutative diagram
\[
\xymatrix{
A \ar[r]^f & B\\
P \ar[u]^g \ar[r]_{\alpha} & Q\ar[u]_h
}
\]
in $\cdg$,
and such that $\alpha(\mathbf{a}) = \mathbf{b}$.
Then, by definition, we have:
\[
\mrm{L}\Lambda_{\bar{\b}}(f^{\flat}(M)) = \mrm{L}\Lambda^Q_{\bar{\b}}(\mrm{R}\opn{Hom}_A(B,M))
\cong \mrm{R}\opn{Hom}_Q(B,\mrm{L}\Lambda_{(\mathbf{b})} ( \mrm{R}h_*(\mrm{R}\opn{Hom}_A(B,M) ) )
\]
Note that
\[
\mrm{R}h_* (\mrm{R}\opn{Hom}_A(B,M)) \cong \mrm{R}\opn{Hom}_P(Q,M),
\]
and that
\[
\mrm{L}\Lambda_{(\mathbf{b})} ( \mrm{R}\opn{Hom}_P(Q,M) ) \cong
\opn{Hom}_{Q^0}(\opn{Tel}(Q^0;\mathbf{b}),\mrm{R}\opn{Hom}_P(Q,M)).
\]
Since there is an isomorphism
$\opn{Tel}(Q^0;\mathbf{b}) \cong \opn{Tel}(P^0;\mathbf{a}) \otimes_{P^0} Q^0$,
using adjunction we have a natural isomorphism
\begin{eqnarray}
\opn{Hom}_{Q^0}(\opn{Tel}(Q^0;\mathbf{b}),\mrm{R}\opn{Hom}_P(Q,M)) \cong\nonumber\\
\mrm{R}\opn{Hom}_P(Q,\opn{Hom}_{P^0}(\opn{Tel}(P^0;\mathbf{a}),M)) \cong\nonumber\\
\mrm{R}\opn{Hom}_P(Q,\mrm{L}\Lambda_{(\mathbf{a})}(M)).\nonumber
\end{eqnarray}
Hence,
\begin{eqnarray}
\mrm{L}\Lambda_{\bar{\b}}(f^{\flat}(M))  
\cong \mrm{R}\opn{Hom}_Q(B,\mrm{R}\opn{Hom}_P(Q,\mrm{L}\Lambda_{(\mathbf{a})}(M)))\cong\nonumber\\
\mrm{R}\opn{Hom}_P(B,\mrm{L}\Lambda_{(\mathbf{a})}(M)) \cong \mrm{R}\opn{Hom}_A(B,\mrm{L}\Lambda_{\bar{\a}}(M)).\nonumber
\end{eqnarray}
\end{proof}

\section{The adic completion functor over commutative DG-rings}

The aim of this section is to extend the adic completion functor from the category of commutative rings to the category of commutative non-positive DG-rings.
We will restrict attnetion only to finitely generated ideals,
although everything in this section holds also for arbitrary ideals. 
The reason for the restriction is that in the next section we will derive this functor,
and that construction only works under the finitely generated assumption.

Fix a commutative noetherian ring $\k$.
We wish to associate to a commutative DG-ring $A$ over $\k$ and a finitely generated ideal $\bar{\a} \subseteq \mrm{H}^0(A)$ a commutative DG-ring over $\k$ which will be called the $\bar{\a}$-adic completion of $A$. To do this, we will lift $\bar{\a}$ to an ideal $\a \subseteq A^0$, such that $\a \cdot \mrm{H}^0(A) = \bar{\a}$,
and then take the $\a$-adic completion of $A$. This can always be done.
However, there is no way to make a canonical choice of such a lift under the assumption that the lift is a finitely generated ideal (since it is possible that the biggest lift, $\pi_A^{-1}(\bar{\a})$ is not a finitely generated ideal). Instead of making a choice of a lift, we consider all liftings, show that completions with respect to all liftings form a directed system, and take its colimit. We will then prove that this operation form a functor $\Lambda: \cdgcont[/\k] \to \cdgcont[/\k]$.

Before performing the construction, we discuss some basic facts about adic completion.

\begin{fact}\label{rem:comp-of-dg}
Let $A$ be a commutative ring,
let $\a\subseteq A$ be a finitely generated ideal,
let $B$ be a commutative DG-ring,
and let $A \to B$ be a map of DG-rings.
For every $n$,
the tensor product $A/\a^n \otimes_A B$ is a commutative DG-ring,
and moreover, these form an inverse system of commutative DG-rings.
Hence, the inverse limit $\varprojlim A/\a^n \otimes_A B$
which is simply $\Lambda_{\a}(B)$,
is also a commutative DG-ring.
If $C$ is another commutative DG-ring,
and $B \to C$ is an $A$-linear map,
the fact that the maps $A/\a^n \otimes_A B \to A/\a^n \otimes_A C$
are maps of commutative DG-rings implies that the map 
\begin{equation}\label{eqn:LambdaDirect}
\Lambda_{\a}(B) \to \Lambda_{\a}(C)
\end{equation}
is also a map of commutative DG-rings.
\end{fact}

\begin{fact}\label{rem:lambdamn}
Given a commutative ring $A$,
and given finitely generated ideals $\a\subseteq \b \subseteq A$,
note that there is a morphism of functors
\begin{equation}\label{eqn:LambdaAtoB}
\Lambda_{\a,\b} : \Lambda_{\a}(-) \to \Lambda_{\b}(-)
\end{equation}
which is induced by the maps $A/{\a}^n \to A/{\b}^n$.
Since for each $n$, the diagram
\[
\xymatrix{
A \ar[r]\ar[d] & A/{\a}^n\ar[ld]\\
A/{\b}^n &
}
\]
is commutative, the diagram
\begin{equation}\label{eqn:AtoLamAdiag}
\xymatrix{
A \ar[r]\ar[d] & \Lambda_{\a}(A)\ar[ld]^{\Lambda_{\a,\b}}\\
\Lambda_{\b}(A) &
}
\end{equation}
is also commutative. If $\c \subseteq A$ is another finitely generated ideal, 
and $\b\subseteq \c$, there is an equality 
\begin{equation}\label{eqn:lambdaABC}
\Lambda_{\a,\c} = \Lambda_{\b,\c} \circ \Lambda_{\a,\b}
\end{equation}
\end{fact}

\begin{fact}\label{fact:diagMixAB}
If $f:A \to B$ is a map of commutative DG-rings,
$\a \subseteq \b \subseteq A^0$ are finitely generated ideals,
and $\c = \a \cdot B^0 \subseteq \d = \b \cdot B^0 \subseteq B^0$,
then the fact that for each $n \in \mathbb{N}$ the diagram
\[
\xymatrix{
A^0/{\a}^n \ar[r]\ar[d] & A^0/{\b}^n \ar[d]\\
A^0/{\a}^n \otimes_{A^0} B^0 \ar[r] & A^0/{\b}^n \otimes_{A^0} B^0
}
\]
is commutative implies that the diagram
\[
\xymatrix{
\Lambda_{\a}(A) \ar[r]^{\Lambda_{\a,\b}}\ar[d]_{\Lambda_{\a}(f)} & \Lambda_{\b}(A)\ar[d]^{\Lambda_{\b}(f)} \\
\Lambda_{\c}(B) \ar[r]_{\Lambda_{\b,\d}} & \Lambda_{\d}(B)
}
\]
of commutative DG-rings is commutative.
\end{fact}

\begin{fact}
For any $(A,\bar{\a}) \in \cdgcont[/\k]$,
let $\opn{Lift}(A,\bar{\a})$ be the set of finitely generated ideals $\a \subseteq A^0$
such that $\a \cdot \mrm{H}^0(A) = \bar{\a}$.
The inclusion relation makes $\opn{Lift}(A,\bar{\a})$ into a partially ordered set. 
Moreover, it is clear that this partially ordered set is a directed set.
If $\mathbf{x}, \mathbf{y} \in \opn{Lift}(A,\bar{\a})$, 
and $\mathbf{y}$ contains $\mathbf{x}$, we write $\mathbf{x} \le \mathbf{y}$.
If $f:(A,\bar{\a}) \to (B,\bar{\b})$ is a morphism in $\cdgcont[/\k]$,
and $\mathbf{x} \in \opn{Lift}(A,\bar{\a})$,
we denote by $f(\mathbf{x})$ the ideal $f(\mathbf{x}) \cdot B^0$.
Hence,
\[
f(\mathbf{x}) \in \opn{Lift}(B,\mrm{H}^0(f)(\bar{\a})\cdot \mrm{H}^0(B)).
\]
\end{fact}

\begin{fact}
Given $(A,\bar{\a}) \in \cdgcont[/\k]$,
we define $\Lambda(A,\bar{\a})$ as follows:
for any $\mathbf{x} \in \opn{Lift}(A,\bar{\a})$,
denote by $\Lambda_{\mathbf{x}}(A)$ the $\mathbf{x}$-adic completion of $A$.
By (\ref{rem:comp-of-dg}), this is a commutative DG-ring.
Given $\mathbf{x} \le \mathbf{y} \in \opn{Lift}(A,\bar{\a})$,
by (\ref{rem:lambdamn}) we have a DG-ring homomorphism $\Lambda_{\mathbf{x},\mathbf{y}}:\Lambda_{\mathbf{x}}(A) \to\Lambda_{\mathbf{y}}(A)$. 
Given $\mathbf{x} \le \mathbf{y} \le \mathbf{z} \in \opn{Lift}(A,\bar{\a})$,
by (\ref{eqn:lambdaABC}) we have that 
$\Lambda_{\mathbf{y},\mathbf{z}} \circ \Lambda_{\mathbf{x},\mathbf{y}} =\Lambda_{\mathbf{x},\mathbf{z}}$.
It follows that 
\[
\{\Lambda_{\mathbf{x}}(A)\}_{\mathbf{x} \in \opn{Lift}(A,\bar{\a})}
\]
is a directed system of commutative DG-rings.
Let us denote its colimit by $\Lambda(A,\bar{\a})$,
and the canonical morphism $\Lambda_{\mathbf{x}}(A) \to \Lambda(A,\bar{\a})$ by $\alpha_{\mathbf{x}}$.
\end{fact}

\begin{fact}\label{fact:mapFromAtoLA}
For any $(A,\bar{\a}) \in \cdgcont[/\k]$,
and any $\mathbf{x} \in \opn{Lift}(A,\bar{\a})$,
there is a natural map $\tau_{\mathbf{x}}:A \to \Lambda_{\mathbf{x}}(A)$.
If $\mathbf{x} \le \mathbf{y} \in \opn{Lift}(A,\bar{\a})$,
By (\ref{eqn:AtoLamAdiag}),
the diagram
\[
\xymatrix{
A \ar[r]^{\tau_{\mathbf{x}}} \ar[d]_{\tau_{\mathbf{y}}} & \Lambda_{\mathbf{x}}(A)\ar[ld]^{\Lambda_{\mathbf{x},\mathbf{y}}}\\
\Lambda_{\mathbf{y}}(A) &
}
\]
is commutative. 
Hence, there is a unique map of DG-rings $\tau_{\bar{\a}}:A \to \Lambda(A,\bar{\a})$,
such that for any $\mathbf{x} \in \opn{Lift}(A,\bar{\a})$,
the diagram
\[
\xymatrix{
A \ar[r]^{\tau_{\mathbf{x}}} \ar[d]_{\tau_{\bar{\a}}} & \Lambda_{\mathbf{x}}(A)\ar[ld]^{\alpha_{\mathbf{x}}}\\
\Lambda(A,\bar{\a}) &
}
\]
is commutative. 
Using the map $\tau_{\bar{\a}}$, 
we give $\Lambda(A,\bar{\a})$ the structure of an element of $\cdgcont[/\k]$
by declaring its ideal of definition to be the finitely generated ideal
\begin{equation}\label{eqn:idealOfDef}
\mrm{H}^0(\tau_{\bar{\a}})(\bar{\a})  \cdot \mrm{H}^0(\Lambda(A,\bar{\a})) \subseteq \mrm{H}^0(\Lambda(A,\bar{\a})).
\end{equation}
\end{fact}

\begin{fact}\label{fact:beforeMap}
Given a map $f:(A,\bar{\a}) \to (B,\bar{\b}) \in \cdgcont[/\k]$,
we wish to define a map $\Lambda(f):\Lambda(A,\bar{\a}) \to \Lambda(B,\bar{\b})$.
In order to do this, we first define for any $\mathbf{x} \in \opn{Lift}(A,\bar{\a})$,
a map $\beta^f_{\mathbf{x}}:\Lambda_{\mathbf{x}}(A) \to \Lambda(B,\bar{\b})$ as follows:
denote by $\bar{\c}$ the finitely generated ideal 
\[
\mrm{H}^0(f)(\bar{\a}) \cdot \mrm{H}^0(B) \subseteq \mrm{H}^0(B).
\]
Then it is clear that $f(\mathbf{x}) \in \opn{Lift}(B,\bar{\c})$.
Choose some $\mathbf{y} \in \opn{Lift}(B,\bar{\b})$ such that $f(\mathbf{x}) \le \mathbf{y}$, 
and let $\beta^f_{\mathbf{x}}$ be the composition
\[
\Lambda_{\mathbf{x}}(A) \to \Lambda_{f(\mathbf{x})}(B) \to \Lambda_{\mathbf{y}}(B) \to \Lambda(B,\bar{\b})
\]
Here, the first map is of the form (\ref{eqn:LambdaDirect}),
the second map is of the form (\ref{eqn:LambdaAtoB}),
and the third map is $\alpha_{\mathbf{y}}$.
\end{fact}

\begin{lem}\label{lem:betaInd}
The map $\beta^f_{\mathbf{x}}:\Lambda_{\mathbf{x}}(A) \to \Lambda(B,\bar{\b})$ defined above is independent of the chosen $\mathbf{y} \in \opn{Lift}(B,\bar{\b})$.
\end{lem}
\begin{proof}
Let $\mathbf{z} \in \opn{Lift}(B,\bar{\b})$ such that $f(\mathbf{x}) \le \mathbf{z}$.
We must show that the composition of
\[
\Lambda_{f(\mathbf{x})}(B) \to \Lambda_{\mathbf{y}}(B) \to \Lambda(B,\bar{\b})
\]
is equal to the composition of 
\[
\Lambda_{f(\mathbf{x})}(B) \to \Lambda_{\mathbf{z}}(B) \to \Lambda(B,\bar{\b})
\]
To see this, choose some $\mathbf{w} \in \opn{Lift}(B,\bar{\b})$ such that $\mathbf{y} \le \mathbf{w}$ and $\mathbf{z} \le \mathbf{w}$.
Then, on the one hand, by (\ref{eqn:lambdaABC}) we have a commutative diagram
\[
\xymatrix{
\Lambda_{f(\mathbf{x})}(B) \ar[r]\ar[d] & \Lambda_{\mathbf{y}}(B) \ar[d]\\
\Lambda_{\mathbf{z}}(B) \ar[r]          & \Lambda_{\mathbf{w}}(B)
}
\]
while on the other hand, by definition of $\Lambda(B,\bar{\b})$, we have a commutative diagram
\[
\xymatrix{
\Lambda_{\mathbf{y}}(B) \ar[r]\ar[rd] & \Lambda_{\mathbf{w}}(B) \ar[d] & \Lambda_{\mathbf{z}}(B) \ar[l]\ar[ld] \\
& \Lambda(B,\bar{\b}) &
}
\]
Hence, both compositions are equal to the composition
\[
\Lambda_{f(\mathbf{x})}(B) \to \Lambda_{\mathbf{w}}(B) \to \Lambda(B,\bar{\b})
\]
so in particular they are equal to each other.
\end{proof}

\begin{lem}\label{lem:beta-are-comp}
Given a map $f:(A,\bar{\a}) \to (B,\bar{\b}) \in \cdgcont[/\k]$,
and given $\mathbf{x} \le \mathbf{y} \in \opn{Lift}(A,\bar{\a})$,
the diagram
\[
\xymatrix{
\Lambda_{\mathbf{x}}(A) \ar[d]_{\beta^f_{\mathbf{x}}} \ar[r]^{\Lambda_{\mathbf{x},\mathbf{y}}} & \Lambda_{\mathbf{y}}(A) \ar[ld]^{\beta^f_{\mathbf{y}}}\\
\Lambda(B,\bar{\b}) &
}
\]
is commutative.
\end{lem}
\begin{proof}
Choose some $\mathbf{z} \in \opn{Lift}(B,\bar{\b})$ such that $f(\mathbf{x}) \le \mathbf{z}$, and some  $\mathbf{w} \in \opn{Lift}(B,\bar{\b})$ such that $f(\mathbf{y}) \le \mathbf{w}$ and moreover $\mathbf{z} \le \mathbf{w}$. Then it follows from (\ref{eqn:lambdaABC}) and (\ref{fact:diagMixAB}) that there is a commutative diagram
\[
\xymatrixrowsep{1pc}
\xymatrix{
\Lambda_{\mathbf{x}}(A) \ar[r] \ar[dd]_{\Lambda_{\mathbf{x},\mathbf{y}}} & \Lambda_{f(\mathbf{x})}(B) \ar[r] \ar[dd] & 
\Lambda_{\mathbf{z}}(B) \ar[dd] \ar[rd] \\
& & & \Lambda(B,\bar{\b}) \\
\Lambda_{\mathbf{y}}(A) \ar[r]        & \Lambda_{f(\mathbf{y})}(B) \ar[r]        &
\Lambda_{\mathbf{w}}(B) \ar[ur] &
}
\] 
By Lemma \ref{lem:betaInd}, the top horizontal composition is equal to $\beta^f_{\mathbf{x}}$,
while the bottom horizontal composition is equal to $\beta^f_{\mathbf{y}}$. Hence, commutativity of this diagram implies the result.
\end{proof}

\begin{fact}
Given a map $f:(A,\bar{\a}) \to (B,\bar{\b}) \in \cdgcont[/\k]$,
continuing (\ref{fact:beforeMap}), we now define a map $\Lambda(f):\Lambda(A,\bar{\a}) \to \Lambda(B,\bar{\b})$ as follows: for any $\mathbf{x} \in \opn{Lift}(A,\bar{\a})$,
we have a map $\beta^f_{\mathbf{x}}:\Lambda_{\mathbf{x}}(A) \to \Lambda(B,\bar{\b})$.
By Lemma \ref{lem:beta-are-comp}, 
these maps are compatible with the transition maps $\Lambda_{\mathbf{x},\mathbf{y}}$.
Hence, by the universal property of colimits, 
there is a unique map of commutative DG-rings $\Lambda(f):\Lambda(A,\bar{\a}) \to \Lambda(B,\bar{\b})$,
such that for each $\mathbf{x} \in \opn{Lift}(A,\bar{\a})$,
the diagram
\begin{equation}\label{eqn:diag-of-lambda}
\xymatrixcolsep{3pc}
\xymatrixrowsep{3pc}
\xymatrix{
\Lambda_{\mathbf{x}}(A) \ar[r]^{\alpha_{\mathbf{x}}}\ar[d]_{\beta^f_{\mathbf{x}}} & \Lambda(A,\bar{\a}) \ar[ld]^{\Lambda(f)}\\
\Lambda(B,\bar{\b}) &
}
\end{equation}
is commutative. 
\end{fact}

\begin{prop}\label{prop:tauIsNatural}
Given a map $f:(A,\bar{\a}) \to (B,\bar{\b}) \in \cdgcont[/\k]$,
the diagram
\[
\xymatrixcolsep{3pc}
\xymatrixrowsep{3pc}
\xymatrix{
A \ar[r]^{\tau_{\bar{\a}}} \ar[d]_f & \Lambda(A,\bar{\a}) \ar[d]^{\Lambda(f)}\\
B \ar[r]_{\tau_{\bar{\b}}} & \Lambda(B,\bar{\b})
}
\]
is commutative.
\end{prop}
\begin{proof}
Take some $\mathbf{x} \in \opn{Lift}(A,\bar{\a})$.
The fact that the map from a complex to its adic completion is natural implies that the diagram
\[
\xymatrixcolsep{3pc}
\xymatrixrowsep{3pc}
\xymatrix{
A \ar[r]^{\tau_{\mathbf{x}}}\ar[d]_{f} & \Lambda_{\mathbf{x}}(A) \ar[d]^{\Lambda_{\mathbf{x}}(f)}\\
B \ar[r]_{\tau_{f(\mathbf{x})}} & \Lambda_{f(\mathbf{x})}(B)
}
\]
is commutative. Take some $\mathbf{y} \in \opn{Lift}(B,\bar{\b})$,
such that $f(\mathbf{x}) \le \mathbf{y}$.
By definition $\beta^f_{\mathbf{x}}$ and of $\Lambda(f)$, the rightmost vertical map in the above diagram fits into the commutative diagram
\[
\xymatrixcolsep{4pc}
\xymatrixrowsep{3pc}
\xymatrix{
\Lambda_{\mathbf{x}}(A) \ar[d]_{\Lambda_{\mathbf{x}}(f)} \ar[rr]^{\alpha_{\mathbf{x}}} \ar[rrd]^{\beta^f_{\mathbf{x}}} & & \Lambda(A,\bar{\a}) \ar[d]^{\Lambda(f)}\\
\Lambda_{f(\mathbf{x})}(B) \ar[r] & \Lambda_{\mathbf{y}}(B) \ar[r] & \Lambda(B,\bar{\b})
}
\]
The result now follows from combining these two commutative diagrams.
\end{proof}

\begin{prop}
Given a map $f:(A,\bar{\a}) \to (B,\bar{\b}) \in \cdgcont[/\k]$,
denoting by $\widehat{\bar{\a}}$ (respectively $\widehat{\bar{\b}}$)
the ideal of definition of $\Lambda(A,\bar{\a})$ (resp. $\Lambda(A,\bar{\b})$) defined in (\ref{eqn:idealOfDef}),
the map $\Lambda(f):\Lambda(A,\bar{\a}) \to \Lambda(B,\bar{\b})$ 
satisfies
\[
\mrm{H}^0(\Lambda(f))(\widehat{\bar{\a}}) \subseteq \widehat{\bar{\b}}.
\]
Hence, $\Lambda(f)$ defines a morphism $(\Lambda(A,\bar{\a}),\widehat{\bar{\a}}) \to (\Lambda(B,\bar{\b}),\widehat{\bar{\b}})$ in $\cdgcont[/\k]$.
\end{prop}
\begin{proof}
Since $f:(A,\bar{\a}) \to (B,\bar{\b})$ is a morphism in $\cdgcont[/\k]$,
by definition we have that $\mrm{H}^0(f)(\bar{\a}) \subseteq \bar{\b}$.
Applying the functor $\mrm{H}^0$ to the commutative diagram of Proposition \ref{prop:tauIsNatural}, we see that the diagram
\[
\xymatrixcolsep{4pc}
\xymatrixrowsep{3pc}
\xymatrix{
\mrm{H}^0(A) \ar[r]^{\mrm{H}^0(\tau_{\bar{\a}})}\ar[d]_{\mrm{H}^0(f)} & \mrm{H}^0(\Lambda(A,\bar{\a})) \ar[d]^{\mrm{H}^0(\Lambda(f))} \\
\mrm{H}^0(B) \ar[r]_{\mrm{H}^0(\tau_{\bar{\b}})}  & \mrm{H}^0(\Lambda(B,\bar{\b}))
}
\]
is commutative, and this implies the result.
\end{proof}

\begin{prop}
Let $f:(A,\bar{\a}) \to (B,\bar{\b})$
and $g:(B,\bar{\b}) \to (C,\bar{\c})$
be two maps in $\cdgcont[/\k]$.
Then one has 
\[
\Lambda(g \circ f) = \Lambda(g) \circ \Lambda(f).
\]
\end{prop}
\begin{proof}
Because the domain of these two maps is $\Lambda(A,\bar{\a})$,
by the universal property of colimits, 
to show they are equal it is enough to show that for any $\mathbf{x} \in \opn{Lift}(A,\bar{\a})$, there is an equality
\[
\Lambda(g) \circ \Lambda(f) \circ \alpha_{\mathbf{x}} = \Lambda(g\circ f)\circ \alpha_{\mathbf{x}}.
\]

Take some $\mathbf{y} \in \opn{Lift}(B,\bar{\b})$ such that $f(\mathbf{x}) \le \mathbf{y}$. Then $g(f(\mathbf{x})) \le g(\mathbf{y})$. Choose some $\mathbf{z} \in \opn{Lift}(C,\bar{\c})$ such that $g(y) \le \mathbf{z}$.

By (\ref{eqn:diag-of-lambda}), the left handside is equal to
$\Lambda(g) \circ \beta^f_{\mathbf{x}}$,
while the right handside is equal to
$\beta^{g \circ f}_{\mathbf{x}}$. 
By the definition of $\beta^f_{\mathbf{x}}$,
the map $\Lambda(g) \circ \beta^f_{\mathbf{x}}$ is equal to the composition
\[
\Lambda_{\mathbf{x}}(A) \to \Lambda_{f(\mathbf{x})}(B) \to \Lambda_{\mathbf{y}}(B) \xrightarrow{\alpha_{\mathbf{y}}} \Lambda(B,\bar{\b}) \xrightarrow{\Lambda(g)} \Lambda(C,\bar{\c})
\]
However, by (\ref{eqn:diag-of-lambda}), $\Lambda(g) \circ \alpha_{\mathbf{y}}$ is equal to $\beta^g_{\mathbf{y}}$, so $\Lambda(g) \circ \beta^f_{\mathbf{x}}$ is also equal to the composition
\begin{equation}\label{eqn:to-show-comp}
\Lambda_{\mathbf{x}}(A) \to \Lambda_{f(\mathbf{x})}(B) \to \Lambda_{\mathbf{y}}(B) \to \Lambda_{g(\mathbf{y})}(C) \to \Lambda_{\mathbf{z}}(C) \to \Lambda(C,\bar{\c})
\end{equation}
Consider the following diagram:
\[ 
\xymatrix{
\Lambda_{\mathbf{x}}(A) \ar[r] \ar[rd] & \Lambda_{g(f(\mathbf{x}))}(C) \ar[r] & \Lambda_{g(\mathbf{y})}(C)   \\
& \Lambda_{f(\mathbf{x})}(B) \ar[u] \ar[r] &  \Lambda_{\mathbf{y}}(B) \ar[u]
}
\]
The left triangle in this diagram is commutative because the $(\mathbf{x})$-adic completion is a functor. The right square in this diagram is commutative by (\ref{fact:diagMixAB}). Thus, this entire diagram is commutative,
which implies that the composition of (\ref{eqn:to-show-comp}) is equal to the composition of
\[
\Lambda_{\mathbf{x}}(A) \to \Lambda_{g(f(\mathbf{x}))}(C) \to \Lambda_{g(\mathbf{y})}(C) \to \Lambda_{\mathbf{z}}(C) \to \Lambda(C,\bar{\c})
\]
By (\ref{eqn:lambdaABC}), this is the same as the composition
\[
\Lambda_{\mathbf{x}}(A) \to \Lambda_{g(f(\mathbf{x}))}(C) \to \Lambda_{\mathbf{z}}(C) \to \Lambda(C,\bar{\c})
\]
which by definition is equal to $\beta^{g\circ f}_{\mathbf{x}}$. This proves the result.
\end{proof}

To summarize the results of this section, we have proved:
\begin{thm}
Let $\k$ be a commutative noetherian ring.
Then the operation
\[
(A,\bar{\a}) \mapsto (\Lambda(A,\bar{\a}),\widehat{\bar{\a}})
\]
defines a functor 
\[
\Lambda: \cdgcont[/\k] \to \cdgcont[/\k],
\]
and the collection of maps $\tau_{\bar{\a}}:A \to \Lambda(A,\bar{\a})$
define a natural transformation
\[
\tau: 1_{\cdgcont[/\k]} \to \Lambda.
\]
If $(A,\a) \in \cdgcont[/\k]$ is an ordinary commutative $\k$-algebra,
then $\Lambda(A,\a)$ is the ordinary $\a$-adic completion of $A$,
$\widehat{\a}$ is its ideal of definition,
and $\tau_{\a}:A\to \Lambda(A,\a)$ is the canonical map from the ring $A$ to its $\a$-adic completion.
\end{thm}

\section{The non-abelian derived functor of adic completion}

In this section we perform the main construction of this paper: a non-abelian derived adic
completion functor. By a non-abelian derived functor we mean a derived functor in the sense of homotopy theory, defined over a category which is not abelian, namely, the category $\cdgcont[/\k]$. We shall freely use the homotopical methods introduced in Section \ref{sec:prel}.

\begin{prop}\label{prop:semi-free-all-wpr}
Let $\k$ be a commutative noetherian ring, 
and let $A = \k[X_i]_{i \in I}$ be a polynomial ring in possibly infinitely many variables over $\k$. Let $\a \subseteq A$ be a finitely generated ideal.
Then $\a$ is weakly proregular.
\end{prop}
\begin{proof}
Let $a_1,\dots,a_n$ be a finite sequence of elements in $A$ which generates the ideal $\a$. Each of these elements is a polynomial over $\k$ in finitely many variables.
Hence, there is a finite set of variables $J \subseteq I$,
such that for each $1 \le i \le n$, $a_i \in \k[X_i]_{i \in J}$.
The ring $B = \k[X_i]_{i \in J}$, being a polynomial ring in finitely many variables over the noetherian ring $\k$ is noetherian. 
Hence, the ideal $(a_1,\dots,a_n)\cdot B \subseteq B$ is weakly proregular.
Since the inclusion map $B \to A$ is flat,
we deduce from \cite[Example 3.0(B)]{AJL1} that the ideal $\a$ is also weakly proregular.
\end{proof}

\begin{lem}\label{lem:sfr-is-wpr}
Let $\k$ be a commutative noetherian ring, 
let $(A,\bar{\a}) \in \cdgcont[/\k]$,
and suppose that $A$ is semi-free over $\k$.
Then for any $\mathbf{x} \in \opn{Lift}(A,\bar{\a})$,
given a finite sequence $\mathbf{a}$ of elements of $A^0$ that generates the ideal $\mathbf{x}$,
we have that $(1_A,A,\mathbf{a})$ is a weakly proregular resolution of $(A,\bar{\a})$.
\end{lem}
\begin{proof}
The fact that $A$ is semi-free over $\k$ implies that $A^0$ is a polynomial ring (in possibly infinitely many variables) over $\k$,
and that $A$ is K-flat over $A^0$.
By Proposition \ref{prop:semi-free-all-wpr}, the ideal $\mathbf{x}$ is weakly proregular.
Hence, the pair $(A,\mathbf{a})$ is a weakly proregular DG-ring,
and hence $(1_A,A,\mathbf{a})$ is a weakly proregular resolution of $(A,\bar{\a})$
\end{proof}

\begin{lem}\label{lem:generalmnlemma}
Let $A$ be a commutative DG-ring,
let $\mathbf{a},\mathbf{b}$ be two finite sequences of elements of $A^0$,
and suppose that $(A,\mathbf{a})$ and $(A,\mathbf{b})$ are weakly proregular DG-rings.
Assume that 
\[
\sqrt{\mathbf{a} \cdot \mrm{H}^0(A)} = \sqrt{\mathbf{b} \cdot \mrm{H}^0(A)}.
\]
Then for any $M \in \opn{DGMod}(A)$,
the natural map
\[
M \to \opn{Hom}_{A^0}(\opn{Tel}(A^0;\mathbf{a}),M)
\]
in $\opn{DGMod}(A)$ is a quasi-isomorphism if and only if the natural map
\[
M \to \opn{Hom}_{A^0}(\opn{Tel}(A^0;\mathbf{b}),M)
\]
in $\opn{DGMod}(A)$ is a quasi-isomorphism.
\end{lem}
\begin{proof}
We may apply the forgetful functor $\opn{DGMod}(A) \to \opn{DGMod}(A^0)$,
and it is enough to check this assertion in $\opn{DGMod}(A^0)$.
Letting $\a := \mathbf{a}\cdot A^0$, $\b := \mathbf{b} \cdot A^0$,
by \cite[Corollary 5.25]{PSY1}, the above assertion is equivalent to showing that 
$M$ is cohomologically $\a$-adically complete if and only if $M$ is cohomologically $\b$-adically complete.
Since $\mrm{L}\Lambda_{\a}$ and $\mrm{L}\Lambda_{\b}$ are idempotent (\cite[Proposition 7.10]{PSY1}),
this is equivalent to showing that there is some isomorphism
\[
\opn{Hom}_{A^0}(\opn{Tel}(A^0;\mathbf{a}),M) \cong \opn{Hom}_{A^0}(\opn{Tel}(A^0;\mathbf{b}),M)
\]
in $\cat{D}(A^0)$,
and this follows from Proposition \ref{prop:form} and Corollary \ref{cor:sqrt},
since $(1_A,A,\mathbf{a})$ is a weakly proregular resolution of $(A,\mathbf{a}\cdot \mrm{H}^0(A) )$, and $(1_A,A,\mathbf{b})$ is a weakly proregular resolution of $(A,\mathbf{b}\cdot \mrm{H}^0(A) )$.	
\end{proof}

\begin{lem}\label{lem:xy-qis}
Let $\k$ be a commutative noetherian ring,
let $A$ be a commutative DG-ring which is semi-free over $\k$,
let $\bar{\a}, \bar{\b} \subseteq \mrm{H}^0(A)$ be finitely generated ideals,
and let $\a \in \opn{Lift}(A,\bar{\a})$ and $\b \in \opn{Lift}(A,\bar{\b})$.
Assume that there is an inclusion $\a \subseteq \b$, and that 
there is an equality $\sqrt{\bar{\a}} = \sqrt{\bar{\b}}$.
Then the morphism
\[
\Lambda_{\a}(A) \xrightarrow{\Lambda_{\a,\b}} \Lambda_{\b}(A)
\]
is a quasi-isomorphism.
\end{lem}
\begin{proof}
Let us choose a finite sequence $\mathbf{a}$ of elements of $A^0$ that generate $\a$,
and then extend it to a finite sequence $\mathbf{b}$ of elements of $A^0$ that generate $\b$. The construction of the natural morphisms
\[
\opn{Hom}_{A^0}(\opn{Tel}(A^0;\mathbf{a}),A) \to \Lambda_{\a}(A), \quad \opn{Hom}_{A^0}(\opn{Tel}(A^0;\mathbf{b}),A) \to \Lambda_{\b}(A)
\]
in \cite[Definition 5.16]{PSY1} shows that the diagram
\[
\xymatrix{
\opn{Hom}_{A^0}(\opn{Tel}(A^0;\mathbf{a}),A) \ar[r]\ar[d]_{\varphi} & \Lambda_{\a}(A) \ar[d]^{\Lambda_{\a,\b}} \\
\opn{Hom}_{A^0}(\opn{Tel}(A^0;\mathbf{b}),A) \ar[r]       & \Lambda_{\b}(A)
}
\]
is commutative. 
Here, the map $\varphi$ is induced from the map
$\opn{Tel}(A^0;\mathbf{a}) \to \opn{Tel}(A^0;\mathbf{b})$,
and that map is in turn induced from the inclusion $\mathbf{a} \subseteq \mathbf{b}$,
and the map $\opn{Tel}(A^0;\mathbf{a}) \to A^0$ from \cite[Equation (5.6)]{PSY1}.
Since by Lemma \ref{lem:sfr-is-wpr}, 
the sequences $\mathbf{a}$ and $\mathbf{b}$ are weakly proregular,
and since $A$ is K-flat over $A^0$, it follows from \cite[Theorem 5.21]{PSY1} that the two horizontal maps in the above diagram are quasi-isomorphisms. 
It is thus enough to show that the map $\varphi$ is a quasi-isomorphism.
Let us denote by $\mathbf{c}$ the finite sequence of elements of $A^0$ such that $\mathbf{b}$ is obtained from concatenating $\mathbf{c}$ at the end of $\mathbf{a}$.
Since $A$ is semi-free over $\k$, $\mathbf{c}$ is also weakly proregular.
Setting $X = \opn{Hom}_{A^0}(\opn{Tel}(A^0;\mathbf{a}),A)$,
we thus need to show that the map $X \to \opn{Hom}_{A^0}(\opn{Tel}(A^0;\mathbf{c}),X)$
is a quasi-isomorphism. This map fits into the commutative diagram
\[
\xymatrix{
X \ar[r] \ar[d] & \opn{Hom}_{A^0}(\opn{Tel}(A^0;\mathbf{c}),X) \ar[ld]\\
\opn{Hom}_{A^0}(\opn{Tel}(A^0;\mathbf{b}),X) &
}
\]
Since $X$ is cohomologically $\bar{\a}$-adically complete,
by Lemma \ref{lem:generalmnlemma}, $X$ is also cohomologically $\bar{\b}$-adically complete, and the the vertical map in this diagram is commutative.
By adjunction,
\[
\opn{Hom}_{A^0}(\opn{Tel}(A^0;\mathbf{c}),X) = \opn{Hom}_{A^0}(\opn{Tel}(A^0;\mathbf{b}),A),
\]
and the latter is cohomologically $\bar{\b}$-adically complete. 
Hence, all maps in the above diagram are quasi-isomorphisms, which proves the claim.
\end{proof}

\begin{lem}\label{lem:lamfisqis}
Let $\k$ be a commutative noetherian ring,
let $f:(A,\bar{\a}) \to (B,\bar{\b})$ be a weak equivalence in $\cdgcont[/\k]$,
let $\a \in \opn{Lift}(A,\bar{\a})$,
and let $\b = f(\mathbf{x})\cdot B^0 \in \opn{Lift}(B,\bar{\b})$.
Let $\mathbf{a}$ be a finite sequence of elements of $A^0$ that generates $\a$,
and let $\mathbf{b} = f(\mathbf{a})$. 
Suppose $(A,\mathbf{a})$ and $(B,\mathbf{b})$ are weakly proregular DG-rings. 
Then the morphism
\[
\Lambda_{\a}(A) \xrightarrow{\Lambda_{\a}(f)} \Lambda_{\b}(B)
\]
is a quasi-isomorphism.
\end{lem}
\begin{proof}
By \cite[Definition 5.16]{PSY1},
there is a commutative diagram in $\opn{DGMod}(A^0)$:
\[
\xymatrix{
\opn{Hom}_{A^0}(\opn{Tel}(A^0;\mathbf{a}),A) \ar[r]\ar[d] & \Lambda_{\a}(A)\ar[d]^{\Lambda_{\a}(f)}\\
\opn{Hom}_{A^0}(\opn{Tel}(A^0;\mathbf{a}),B) \ar[r]_{\alpha} & \Lambda_{\a}(B) = \Lambda_{\b}(B)
}
\]
Moreover, since $A$ is K-flat over $A^0$, 
by \cite[Corollary 5.23]{PSY1}, the top horizontal map is a quasi-isomorphism.
Since $A \to B$ is a quasi-isomorphism, 
and since $\opn{Tel}(A^0;\mathbf{a})$ is K-projective over $A^0$,
the left vertical map is a quasi-isomorphism.
Similarly, there is a quasi-isomorphism
\[
\beta: \opn{Hom}_{B^0}(\opn{Tel}(B^0;\mathbf{b}),B) \to \Lambda_{\b}(B)
\]
in $\opn{DGMod}(B^0)$.
Letting $\opn{Rest}_{f^0}:\opn{DGMod}(B^0) \to \opn{DGMod}(A^0)$ be the forgetful functor along the map $f^0:A^0 \to B^0$,
the fact that $\opn{Tel}(A^0;\mathbf{a}) \otimes_{A^0} B^0 = \opn{Tel}(B^0;\mathbf{b})$ and the construction of $\beta$ in \cite{PSY1},
implies that $\opn{Rest}_{f^0}(\beta) = \alpha$, 
so that $\alpha$ is also a quasi-isomorphism. 
Hence, $\Lambda_{\a}(f)$ is a quasi-isomorphism.
\end{proof}

The next lemma is the key to deriving the functor $\Lambda$.
\begin{lem}\label{lem:key-der}
Let $\k$ be a commutative noetherian ring,
let $f:(C_1,\bar{\a}) \to (C_2,\bar{\b})$ be a weak equivalence in $\cdgcont[/\k]$,
and suppose that $C_1,C_2$ are semi-free over $\k$.
Then the morphism
\[
\Lambda(f):\Lambda(C_1,\bar{\a}) \to \Lambda(C_2,\bar{\b})
\]
is also a weak equivalence in $\cdgcont[/\k]$,
\end{lem}
\begin{proof}
Given any $\mathbf{x} \le \mathbf{y} \in \opn{Lift}(C_1,\bar{\a})$,
it follows from Lemma \ref{lem:sfr-is-wpr} and Lemma \ref{lem:xy-qis}
that the map $\Lambda_{\mathbf{x},\mathbf{y}}:\Lambda_{\mathbf{x}}(C_1) \to \Lambda_{\mathbf{y}}(C_1)$ is a quasi-isomorphism.
In other words, in the directed system
\[
\{\mrm{H}(\Lambda_{\mathbf{x}}(C_1))\}_{\mathbf{x} \in \opn{Lift}(C_1,\bar{\a})}
\]
all the transition maps are isomorphisms. 
Hence, for each $\mathbf{x} \in \opn{Lift}(C_1,\bar{\a})$,
the map 
\begin{equation}\label{eqn:trans-are-iso}
\mrm{H}(\alpha_{\mathbf{x}}):\mrm{H}(\Lambda_{\mathbf{x}}(C_1)) \to \mrm{H}(\Lambda_{\a}(C_1))
\end{equation}
is an isomorphism. 
In other words, the map $\alpha_{\mathbf{x}}$ is a quasi-isomorphism.
Similarly, for any $\mathbf{x} \le \mathbf{y} \in \opn{Lift}(C_2,\bar{\b})$,
the maps $\Lambda_{\mathbf{x},\mathbf{y}}:\Lambda_{\mathbf{x}}(C_2) \to \Lambda_{\mathbf{y}}(C_2)$ and $\alpha_{\mathbf{y}}:\Lambda_{\mathbf{y}}(C_2) \to \Lambda(C_2,\bar{\b})$ are quasi-isomorphisms.

Choose some $\mathbf{x} \in \opn{Lift}(C_1,\bar{\a})$,
and let $\mathbf{y} \in \opn{Lift}(C_2,\bar{\b})$ be such that $f(\mathbf{x}) \le \mathbf{y}$. 
By definition, the map $\Lambda(f)$ fits into the commutative diagram
\[
\xymatrixcolsep{4pc}
\xymatrixrowsep{3pc}
\xymatrix{
\Lambda_{\mathbf{x}}(C_1) \ar[d]_{\Lambda_{\mathbf{x}}(f)} \ar[rr]^{\alpha_{\mathbf{x}}}  & & \Lambda(C_1,\bar{\a}) \ar[d]^{\Lambda(f)}\\
\Lambda_{f(\mathbf{x})}(C_2) \ar[r]_{\Lambda_{f(\mathbf{x}),\mathbf{y}}} & \Lambda_{\mathbf{y}}(C_2) \ar[r]_{\alpha_{\mathbf{y}}} & \Lambda(C_2,\bar{\b})
}
\]
We have seen that $\alpha_{\mathbf{x}}$ and $\alpha_{\mathbf{y}}$ are quasi-isomorphism.
The map $\Lambda_{\mathbf{x}}(f)$ is a quasi-isomorphism by Lemma \ref{lem:lamfisqis},
and the map $\Lambda_{f(\mathbf{x}),\mathbf{y}}$ is a quasi-isomorphism by Lemma \ref{lem:xy-qis}.
Hence, the map $\Lambda(f)$ is also a quasi-isomorphism.
\end{proof}

\begin{thm}\label{thm:main}
Let $\k$ be a commutative noetherian ring,
and let $Q_{\k}:\cdgcont[/\k] \to \ho(\cdgcont[/\k])$ be the localization functor.
There is a functor
\[
\mrm{L}\Lambda:\ho(\cdgcont[/\k]) \to \ho(\cdgcont[/\k]),
\]
a natural transformation
\[
\Tau:1_{\ho(\cdgcont[/\k])} \to \mrm{L}\Lambda,
\]
and a natural transformation
\[
\eta: \mrm{L}\Lambda\circ Q_{\k} \to Q_{\k} \circ \Lambda
\]
which satisfy the following properties:
\begin{enumerate}
\item For any $(A,\bar{\a}) \in \cdgcont[/\k]$,
and any weakly proregular resolution 
\[
P = (g,P_A,\mathbf{a})
\] 
of $(A,\bar{\a})$, there is an isomorphism $\varphi_P: \Lambda_{\mathbf{a}}(P) \to \mrm{L}\Lambda(A,\bar{\a})$ in $\ho(\cdgcont[/\k])$ making the diagram
\[
\xymatrix{
P_A \ar[r]^g \ar[d]_{\tau_{\mathbf{a}}} & A\ar[d]^{\Tau_A}\\
\Lambda_{\mathbf{a}}(P_A) \ar[r]_{\varphi_P} & \mrm{L}\Lambda(A,\bar{\a})
}
\]
in $\ho(\cdgcont[/\k])$ commutative.

\item Given a morphism $f:(A,\bar{\a}) \to (B,\bar{\b})$ in $\cdgcont[/\k]$,
and weakly proregular resolutions $P = (g,P_A,\mathbf{a})$, $P' = (h,P_B,\mathbf{b})$ of $(A,\bar{\a})$ and $(B,\bar{\b})$ respectively,
for any map $P_f:P_A \to P_B$ in $\cdg_{\k}$ such that $P_f(\mathbf{a}) \subseteq \mathbf{b}$, and such that the diagram
\[
\xymatrix{
P_A \ar[r]^{P_f}\ar[d]_{g} & P_B \ar[d]^{h}\\
A \ar[r]_f & B
}
\]
is commutative, 
the diagram 
\begin{equation}\label{eqn:P2}
\xymatrixcolsep{4pc}
\xymatrix{
\Lambda_{\mathbf{a}}(P_A) \ar[r]^{\Lambda(P_f)} \ar[d]_{\varphi_P} & \Lambda_{\mathbf{b}}(P_B)\ar[d]^{\varphi_{P'}}\\
\mrm{L}\Lambda(A,\bar{\a}) \ar[r]_{\mrm{L}\Lambda(f)} & \mrm{L}\Lambda(B,\bar{\b})
}
\end{equation}
in $\ho(\cdgcont[/\k])$ is also commutative.
\end{enumerate}

\end{thm}
\begin{proof}
Given $(A,\bar{\a}) \in \cdgcont[/\k]$,
define
\[
\mrm{L}\Lambda(A,\bar{\a}) := \Lambda(\opn{SF}_{\k}(A,\bar{\a})).
\]
Here $\opn{SF}_{\k}$ is the functor from (\ref{eqn:liftedSF}).
Given a morphism $f:(A,\bar{\a}) \to (B,\bar{\b})$ in $\cdgcont[/\k]$,
let $\mrm{L}\Lambda(f) := \Lambda(\opn{SF}_{\k}(f))$. 
These definitions give rise to a functor $\mrm{L}\Lambda:\cdgcont[/\k] \to \ho(\cdgcont[/\k])$. 
If $f:(A,\bar{\a}) \to (B,\bar{\b})$ is a weak equivalence, 
it follows from Lemma \ref{lem:key-der} that $\mrm{L}\Lambda(f)$ is also a weak equivalence. Hence, there is an induced functor 
\[
\ho(\cdgcont[/\k]) \to \ho(\cdgcont[/\k]),
\]
which we also denote by $\mrm{L}\Lambda$.

Applying the functor $\Lambda$ to the natural transformation $\opn{SF}_{\k} \to Q_{\k}$ induces the natural transformation
$\eta: \mrm{L}\Lambda\circ Q_{\k} \to Q_{\k} \circ \Lambda$.

To define $\Tau$, note first that since $\opn{SF}_{\k}$ preserves weak equivalences,
it can be lifted to a functor $\opn{SF}_{\k} : \ho(\cdgcont[/\k]) \to \ho(\cdgcont[/\k])$,
and a natural isomorphism $\iota: \opn{SF}_{\k} \to 1_{\ho(\cdgcont[/\k])}$.
The arrow $\Tau_{(A,\bar{\a})}:A \to \mrm{L}\Lambda(A,\bar{\a})$ is now the unique dotted map making the diagram
\[
\xymatrix{
\opn{SF}(A,\bar{\a}) \ar[r]_{\iota_{(A,\bar{\a})}}^{\cong} \ar[d]_{\tau} & (A,\bar{\a}) \ar@{-->}[ld]\\
\mrm{L}\Lambda(A,\bar{\a}) = \Lambda(\opn{SF}(A,\bar{\a})) &
}
\]
in $\ho(\cdgcont[/\k])$ commutative. 
Since $\iota^{-1}$ and $\tau$ are natural morphisms, 
we deduce that $\Tau:1_{\ho(\cdgcont[/\k])} \to \mrm{L}\Lambda$ is also a natural morphism.
\noindent
It remains to verify that $\mrm{L}\Lambda$ and $\Tau$ satisfy the properties (1) and (2) above:
\begin{enumerate}[wide, labelwidth=!, labelindent=0pt]
\item Given $(A,\bar{\a}) \in \cdgcont[/\k]$,
and given a weakly proregular resolution $P = (g,P_A,\mathbf{a})$
of $(A,\bar{\a})$, the semi-free resolution functor induces a commutative diagram in $\ho(\cdgcont[/\k])$:
\[
\xymatrixcolsep{4pc}
\xymatrix{
(P_A,\bar{\a}) \ar[r]^g & (A,\bar{\a})\\
\opn{SF}_{\k}(P_A,\bar{\a}) \ar[u]\ar[r]_{\opn{SF}_{\k}(g)} & \opn{SF}_{\k}(A,\bar{\a}) \ar[u]
}
\]
in which all maps are isomorphisms,
and the vertical maps are surjective. 
In particular, since the map 
\begin{equation}\label{eqn:sf-map}
{SF}_{\k}(P_A)^0 \to P_A^0
\end{equation}
is surjective, we can choose a finite sequence of elements $\mathbf{a}'$ of ${SF}_{\k}(P_A)^0$, such that its image under the map (\ref{eqn:sf-map}) is equal to $\mathbf{a} \subseteq P_A^0$. As $(\opn{SF}_{\k}(P_A),\mathbf{a}')$ and $(P_A,\mathbf{a})$ are both weakly proregular DG-rings and the map $\opn{SF}_{\k}(P_A) \to P_A$ is a quasi-isomorphism,
it follows from Lemma \ref{lem:lamfisqis} that the map
$\Lambda_{\mathbf{a}'}(\opn{SF}_{\k}(P_A)) \to \Lambda_{\mathbf{a}}(P_A)$ is a quasi-isomorphism. 
Let $\mathbf{x}$ be the ideal in ${SF}_{\k}(P_A)^0$ generated by $\mathbf{a}'$.
Then $\mathbf{x} \in \opn{Lift}({SF}_{\k}(P_A),\bar{\a})$.
By definition,
\[
\Lambda_{\mathbf{a}'}(\opn{SF}_{\k}(P_A)) = \Lambda_{\mathbf{x}}(\opn{SF}_{\k}(P_A)).
\]
It follows from (\ref{eqn:trans-are-iso}) that the canonical map
\[
\Lambda_{\mathbf{x}}(\opn{SF}_{\k}(P_A)) \to \Lambda(\opn{SF}_{\k}(P_A),\bar{\a}) = \mrm{L}\Lambda(P_A,\bar{\a})
\]
is an isomorphism. We thus have isomorphisms
\[
\Lambda_{\mathbf{a}}(P_A) \leftarrow \Lambda_{\mathbf{x}}(\opn{SF}_{\k}(P_A)) \to \mrm{L}\Lambda(P_A,\bar{\a}) \xrightarrow{\mrm{L}\Lambda(g)} \mrm{L}\Lambda(A,\bar{\a})
\]
So we obtain an isomorphism
$\varphi_P: \Lambda_{\mathbf{a}}(P) \to \mrm{L}\Lambda(A,\bar{\a})$ in $\ho(\cdgcont[/\k])$.
Since $\tau$ and $\Tau$ are natural transformations,
it is clear that the diagram
\[
\xymatrix{
P_A \ar[r]^g \ar[d]_{\tau_{\mathbf{a}}} & A\ar[d]^{\Tau_A}\\
\Lambda_{\mathbf{a}}(P_A) \ar[r]_{\varphi_P} & \mrm{L}\Lambda(A,\bar{\a})
}
\]
commutes.
\item Consider a morphism $f:(A,\bar{\a}) \to (B,\bar{\b})$ in $\cdgcont[/\k]$,
and weakly proregular resolutions $P = (g,P_A,\mathbf{a})$, $P' = (h,P_B,\mathbf{b})$ of $(A,\bar{\a})$ and $(B,\bar{\b})$ respectively as in (2) in the statement of the theorem.
Let $\mathbf{x}$ be an ideal in $\opn{SF}_{\k}(P_A)^0$ and $\mathbf{a}'$ a sequence of elements of $\opn{SF}_{\k}(P_A)^0$ that generates it, as in the proof of (1).
We know that there is a map $\opn{SF}_{\k}(P_f)$ making the diagram
\begin{equation}\label{eqn:PFuncDiag}
\xymatrixcolsep{4pc}
\xymatrix{
P_A \ar[r]^{P_f} & P_B\\
\opn{SF}_{\k}(P_A) \ar[u] \ar[r]_{\opn{SF}_{\k}(P_f)} & \opn{SF}_{\k}(P_B) \ar[u]
}
\end{equation}
A-priori, the map $\opn{SF}_{\k}(P_f)$ is only defined in the homotopy category $\ho(\cdgcont[/\k])$. However, the construction of this map shows that it is (the image under the localization map of ) an honest map of commutative DG-rings.
In particular, it makes sense to look at the sequence $\mathbf{b}' = \opn{SF}_{\k}(P_f)(\mathbf{a}')$, and the ideal $\mathbf{y}$ in $\opn{SF}_{\k}(P_B)^0$ generated by it.
Commutativity of (\ref{eqn:PFuncDiag}) implies that $h(\mathbf{b}') = \mathbf{b}$.

Functoriality of $\Lambda$ now implies that there is a commutative diagram 
\[
\xymatrixcolsep{4pc}
\xymatrix{
\Lambda_{\mathbf{a}}(P_A) \ar[d]_{\Lambda(P_f)} & \Lambda_{\mathbf{x}}(\opn{SF}_{\k}(P_A)) \ar[l]\ar[r] \ar[d]^{\Lambda(\opn{SF}_{\k}(P_f))} & \mrm{L}\Lambda(P_A,\bar{\a}) \ar[r]^{\mrm{L}\Lambda(g)} \ar[d]_{\mrm{L}\Lambda(P_f)} & \mrm{L}\Lambda(A,\bar{\a}) \ar[d]^{\mrm{L}\Lambda(f)} \\
\Lambda_{\mathbf{b}}(P_B) & \Lambda_{\mathbf{y}}(\opn{SF}_{\k}(P_B)) \ar[l]\ar[r] & \mrm{L}\Lambda(P_B,\bar{\b}) \ar[r]^{\mrm{L}\Lambda(h)} & \mrm{L}\Lambda(B,\bar{\b})
}
\]
in which all horizontal maps are isomorphisms.
Since the horizontal compositions are by definition $\varphi_P$ and $\varphi_{P'}$,
this establishes the commutativity of (\ref{eqn:P2}).
\end{enumerate}
\end{proof}

\begin{cor}
Let $A$ be a commutative ring,
and let $\a\subseteq A$ be a weakly proregular ideal.
\begin{enumerate}
\item There is a commutative diagram
\[
\xymatrix{
& A\ar[ld]_{\tau^A_{\a}}\ar[rd]^{\Tau_{(A,\a)}} &\\
\Lambda_{\a}(A) \ar[rr] & & \mrm{L}\Lambda(A,\a)
}
\]
in $\ho(\cdg)$, in which the horizontal map is an isomorphism.
\item If $B$ is another commutative ring,
$\b\subseteq B$ is a weakly proregular ideal,
$f:A \to B$ is a ring map, 
and $f(\a)\subseteq B = \b$,
then there is a commutative diagram
\[
\xymatrixcolsep{4pc}
\xymatrix{
\Lambda_{\a}(A) \ar[r]^{\widehat{f}}\ar[d] & \Lambda_{\b}(B)\ar[d]\\
\mrm{L}\Lambda(A,\a) \ar[r]_{\mrm{L}\Lambda(f)} & \mrm{L}\Lambda(B,\b)
}
\]
in $\ho(\cdg)$, in which the vertical maps are isomorphisms.
\end{enumerate}
\end{cor}
\begin{proof}
This is because $(A,\a)$ (respectively $(B,\b)$) is a weakly proregular resolution of itself.
\end{proof}

\begin{dfn}
Let $\k$ be a noetherian ring.
Given a commutative DG-ring $A$ over $\k$,
and a finitely generated ideal $\bar{\a}\subseteq \mrm{H}^0(A)$,
we say that $A$ is cohomologically $\bar{\a}$-adically complete (or derived $\bar{\a}$-adically complete) if the map
\[
\Tau_{(A,\bar{\a})} : A \to \mrm{L}\Lambda(A,\bar{\a})
\]
in $\ho(\cdgcont[/\k])$ is an isomorphism.
\end{dfn}

\begin{prop}
For any $(A,\bar{\a})$ in $\cdgcont[/\k]$,
the morphism
$\Tau_{\mrm{L}\Lambda(A,\bar{\a})}$ is an isomorphism.
\end{prop}
\begin{proof}
Choose some $a_1,\dots,a_n \in A^0$,
such that $(a_1,\dots,a_n)\cdot \mrm{H}^0(A) = \bar{\a}$, 
and consider $A$ as a DG-ring over $\k[x_1,\dots,x_n]$ by setting $x_i \mapsto a_i$.
Factor the structure map $\k[x_1,\dots,x_n] \to A$ as
$\k[x_1,\dots,x_n] \to P \to A$ such that $P \to A$ is surjective quasi-isomorphism and $P$ semi-free commutative DG-ring over $\k[x_1,\dots,x_n]$.
Similarly, factor the structure map $\k[x_1,\dots,x_n] \to \widehat{P} := \Lambda_{(x_1,\dots,x_n)}(P)$,
as $\k[x_1,\dots,x_n] \to Q \to \widehat{P}$, where $Q$ is semi-free a commutative DG-ring over $\k[x_1,\dots,x_n]$, and $Q \to \widehat{P}$ is a surjective quasi-isomorphism.
By Theorem \ref{thm:main},
the map $\Tau_{\mrm{L}\Lambda(A,\bar{\a})}$ is an isomorphism if and only if the map
\[
Q \to \Lambda_{(x_1,\dots,x_n)}(Q)
\]
is a quasi-isomorphism.
By Corollay \ref{cor:lift-hom},
there is a map $\alpha:P \to Q$ making the diagram
\[
\xymatrix{
& P\ar[ld]_{\alpha}\ar[d]\\
Q \ar[r] & \widehat{P}
}
\]
in $\cat{D}(\k[x_1,\dots,x_n])$ commutative.
Let $\b$ be the ideal in $P^0$ generated by the image of $(x_1,\dots,x_n)$.
Using the map $\alpha$, 
we see that $Q \cong \Lambda_{\b}(P)$ in $\cat{D}(P^0)$,
and since $P$ is K-flat over $P^0$,
we deduce that $Q$ is cohomologically $\b$-adically complete.
Let $\c$ be the ideal in $Q^0$ generated by the image of $(x_1,\dots,x_n)$,
and noting that $\c = \alpha(\b)\cdot Q^0$,
we deduce by \cite[Theorem 6.5]{PSY1}
that $Q$ is cohomologically $\c$-adically complete.
Since $Q$ is K-flat over $Q^0$,
this implies that the completion map $Q \to \Lambda_{\c}(Q)$ is a quasi-isomorphism,
and this completes the proof.
\end{proof}

\begin{prop}\label{prop:com-is-dercom}
Let $A$ be a commutative ring,
let $\a\subseteq A$ be a finitely generated ideal,
and assume that $A$ is $\a$-adically complete.
Then $A$ is cohomologically $\a$-adically complete.
\end{prop}
\begin{proof}
Assume $\a=(a_1,\dots,a_n)$, 
and use these elements to give $A$ the structure of a $\k[x_1,\dots,x_n]$-algebra.
Let $\k[x_1,\dots,x_n] \to C \to A$ be a factorization of the structure map,
such that $C$ is a semi-free over $\k[x_1,\dots,x_n]$, and the map $C \to A$ is a quasi-isomorphism.
By Theorem \ref{thm:main}, is enough to show that the map $C \to \Lambda_{(x_1,\dots,x_n)}(C)$ is a quasi-isomorphism.
Since $C$ is K-flat over $\k[x_1,\dots,x_n]$, 
this is the case if and only if $C$ is cohomologically $(x_1,\dots,x_n)$-adically complete in $\cat{D}(\k[x_1,\dots,x_n])$
if and only if $A$ is cohomologically $(x_1,\dots,x_n)$-adically complete in $\cat{D}(\k[x_1,\dots,x_n])$.
The fact that $A$ is $\a$-adically complete implies that it is $(x_1,\dots,x_n)$-adically complete,
and since $\k[x_1,\dots,x_n]$ is a noetherian ring, 
by \cite[Theorem 1.21]{PSY2}, $A$ is cohomologically $(x_1,\dots,x_n)$-adically complete,
proving the claim.
\end{proof}

\begin{prop}\label{prop:lambda_of_wpr}
Let $\k$ be a noetherian ring, 
and let $(A,\bar{\a}) \in \cdgcont[/\k]$.
Assume that as an ideal of the ring $\mrm{H}^0(A)$,
the ideal $\bar{\a}$ is weakly proregular.
Then there is a $\k$-algebra isomorphism
\[
\mrm{H}^0(\mrm{L}\Lambda(A,\bar{\a})) \cong \Lambda_{\bar{\a}}(\mrm{H}^0(A)).
\]
\end{prop}
\begin{proof}
Let $\bar{\mathbf{a}} = (\bar{a}_1,\dots,\bar{a}_n)$ be a finite set of elements of $\mrm{H}^0(A)$ that generates $\bar{\a}$,
and let $\mathbf{a} = (a_1,\dots,a_n)$ be some lifts of these elements to $A^0$.
Using these sets of elements, 
we give $A$ and $\mrm{H}^0(A)$ a $\k[x_1,\dots,x_n]$ structure.
It follows that the map $A \to \mrm{H}^0(A)$ is $\k[x_1,\dots,x_n]$-linear.
Let $C_A \to A$ and $C_{\bar{A}} \to \mrm{H}^0(A)$ be $\k[x_1,\dots,x_n]$-linear surjective quasi-isomorphisms such that $C_A, C_{\bar{A}}$ are semi-free over $\k[x_1,\dots,x_n]$,
and such that there is a commutative diagram
\begin{equation}\label{eqn:diag-of-bara}
\xymatrix{
A \ar[r] & \mrm{H}^0(A)\\
C_A\ar[u]\ar[r]_f & C_{\bar{A}}\ar[u]
}
\end{equation}
in $\cdg_{\k[x_1,\dots,x_n]}$.
By Theorem \ref{thm:main}, there is a commutative diagram
\[
\xymatrix{
\widehat{C}_A \ar[r]^{\widehat{f}} \ar[d] & \widehat{C}_{\bar{A}} \ar[d]\\
\mrm{L}\Lambda(A,\bar{\a}) \ar[r] & \mrm{L}\Lambda(\mrm{H}^0(A),\bar{\a})
}
\]
in $\ho(\cdg)$, such that the vertical maps are isomorphisms.
Since $\bar{\a} \subseteq \mrm{H}^0(A)$ is weakly proregular,
we have that 
\[
\mrm{L}\Lambda(\mrm{H}^0(A),\bar{\a}) \cong \Lambda_{\bar{\a}}(\mrm{H}^0(A)).
\]
Hence, 
\begin{equation}\label{eqn:H0func}
\mrm{H}^0(\widehat{C}_{\bar{A}}) \cong \Lambda_{\bar{\a}}(\mrm{H}^0(A)).
\end{equation}
We may extract the following commutative diagram with exact rows from (\ref{eqn:diag-of-bara}):
\[
\xymatrix{
C_A^{-1} \ar[r]\ar[d] & C_A^0\ar[r]\ar[d] & \mrm{H}^0(A) \ar[d] \ar[r] & 0\\
C_{\bar{A}}^{-1} \ar[r] & C_{\bar{A}}^0\ar[r] & \mrm{H}^0(A) \ar[r] & 0
}
\]
Considered over the noetherian ring $\k[x_1,\dots,x_n]$,
we see that
\[
\xymatrix{
C_A^{-1} \ar[r]\ar[d] & C_A^0\ar[r]\ar[d] & 0\\
C_{\bar{A}}^{-1} \ar[r] & C_{\bar{A}}^0\ar[r] &  0
}
\]
are both truncations of projective resolutions of  the $\k[x_1,\dots,x_n]$-module $\mrm{H}^0(A)$.
Hence, applying the functor $\Lambda_{(x_1,\dots,x_n)}$ to this diagram and taking cohomology,
at both rows we must obtain the classical $0$-th derived functor $\mrm{L}^0\Lambda_{(x_1,\dots,x_n)}(\bar{A})$,
which by (\ref{eqn:H0func}) is isomorphic to $\Lambda_{\bar{\a}}(\mrm{H}^0(A))$.
Thus, we deduce that 
\[
\mrm{H}^0(\widehat{C}_A) \cong \Lambda_{\bar{\a}}(\mrm{H}^0(A)),
\]
so the isomorphism $\widehat{C}_A \to \mrm{L}\Lambda(A,\bar{\a})$ implies the required result.
\end{proof}

The functor $\mrm{R}(-)_*$ was defined in (\ref{eqn:Rf}).
\begin{rem}
Given a commutative ring $A$, and a finitely generated ideal $\a\subseteq A$, 
the canonical map $A \to \Lambda_{\a}(A)$ induces a forgetful functor $\cat{D}(\Lambda_{\a}(A)) \to \cat{D}(A)$.
Similarly, for $(A,\bar{\a}) \in \ho(\cdgcont)$,
we have seen that there is a map $\Tau_A:(A,\bar{\a}) \to \mrm{L}\Lambda(A,\bar{\a})$.
Using it, we obtain an analogue "forgetful functor"
\[
\mrm{R}(\Tau_A)_* : \cat{D}(\mrm{L}\Lambda(A,\bar{\a})) \to \cat{D}(A).
\]
\end{rem}

Given a DG-ring $A$, we finish this section with the next proposition, which describes the structure of its derived completion as a DG-module over $A$.

\begin{prop}
For any $(A,\bar{\a}) \in \ho(\cdgcont)$,
there is an isomorphism
\[
\mrm{R}(\Tau_A)_*(\mrm{L}\Lambda(A,\bar{\a})) \cong \mrm{L}\Lambda_{\bar{\a}}(A)
\]
in $\cat{D}(A)$.
\end{prop}
\begin{proof}
Let $P = (f,P_A,\mathbf{a})$ be a weakly proregular resolution of $(A,\bar{\a})$.
Then
\[
\mrm{R}(\Tau_A)_*(\mrm{L}\Lambda(A,\bar{\a})) \cong \widehat{P}_A\otimes^{\mrm{L}}_{P_A} A \cong \mrm{L}\Lambda_{\bar{\a}}^{P}(A) \cong \mrm{L}\Lambda_{\bar{\a}}(A).
\]
\end{proof}

\section{The functors $\mrm{R}\widehat{\Gamma}_{\bar{\a}}, \mrm{L}\widehat{\Lambda}_{\bar{\a}}$}\label{sec:RRLL}

Given a commutative ring $A$, a finitely generated ideal $\a\subseteq A$,
and an $A$-module $M$,
recall that the modules $\Gamma_{\a}(M)$, 
$\Lambda_{\a}(M)$ carry naturally the structure of $\Lambda_{\a}(A)$-modules. 
We denote these $\Lambda_{\a}(A)$-modules by $\widehat{\Gamma}_{\a}(M)$ 
and $\widehat{\Lambda}_{\a}(M)$. 
These are additive functors 
\[
\widehat{\Gamma}_{\a}(-), \widehat{\Lambda}_{\a}(-) : \opn{Mod} (A) \to \opn{Mod} (\Lambda_{\a}(A)).
\]
They have derived functors
\[
\mrm{R}\widehat{\Gamma}_{\a}(-), \mrm{L}\widehat{\Lambda}_{\a}(-) : \cat{D}(A) \to \cat{D}(\Lambda_{\a}(A)).
\]
If $F:\cat{D}(\Lambda_{\a}(A)) \to \cat{D}(A)$ is the forgetful functor,
it is clear that 
\begin{equation}\label{eqn:RWForget}
F \circ \mrm{R}\widehat{\Gamma}_{\a}(-) \cong \mrm{R}\Gamma_{\a}(-), \quad F \circ \mrm{L}\widehat{\Lambda}_{\a}(-) \cong \mrm{L}\Lambda_{\a}(-).
\end{equation}

We now define analogue functors in $\ho(\cdgcont)$.
We will make use of the functors $\mrm{R}(-)_*, \mrm{L}(-)^*$ and $\mrm{R}(-)^{\flat}$ which were defined in (\ref{eqn:Rf}).
We begin by imitating Definition \ref{dfn:RGP}.

\begin{dfn}
Let $(A,\bar{\a}) \in \cdgcont$,
and let $P = (f,B,\mathbf{b})$ be a weakly proregular resolution of $(A,\bar{\a})$.
Let $\b$ be the ideal in $B^0$ generated by $\mathbf{b}$,
and let $\varphi_P: \Lambda_{\b}(P) \iso \mrm{L}\Lambda(A,\bar{\a})$ be the isomorphism from Theorem \ref{thm:main}.
The operations 
\[
\widehat{\Gamma}_{\b}(-) := \varinjlim_n \opn{Hom}_{B^0}(B^0/{\b}^n,-), \quad \widehat{\Lambda}_{\b}(-) := \varprojlim_n B^0/{\b}^n \otimes_{B^0} -
\]
define additive functors $\opn{DGMod}(B) \to \opn{DGMod}(\Lambda_{\b}(B))$.
Let
\[
\mrm{R}\widehat{\Gamma}_{\b}, \mrm{L}\widehat{\Lambda}_{\b} : \cat{D}(B) \to \cat{D}(\Lambda_{\b}(B))
\]
be their derived functors, and set
\[
\mrm{R}\widehat{\Gamma}_{\bar{\a}}^P(-) :=  \mrm{L}(\varphi_P)^* \mrm{R}\widehat{\Gamma}_{\b}(\mrm{R}f_*(-))
\]
and
\[
\mrm{L}\widehat{\Lambda}_{\bar{\a}}^P(-) :=  \mrm{R}(\varphi_P)^{\flat} \mrm{L}\widehat{\Lambda}_{\b}(\mrm{R}f_*(-)).
\]
These are functors $\cat{D}(A) \to \cat{D}(\mrm{L}\Lambda(A,\bar{\a}))$.
\end{dfn}

\begin{lem}\label{lem:RWLW}
Let $(B,\mathbf{b})$ be a weakly proregular DG-ring.
Then there are isomorphisms
\[
\mrm{R}\widehat{\Gamma}_{\b}(-) \cong \Lambda_{\b}(B) \otimes^{\mrm{L}}_B \mrm{R}\Gamma_{\b}(-)
\]
and
\[
\mrm{L}\widehat{\Lambda}_{\b}(-) \cong \mrm{R}\opn{Hom}_B(\Lambda_{\b}(B),\mrm{L}\Lambda_{\b}(-))
\]
of functors
\[
\cat{D}(B) \to \cat{D}(\Lambda_{\b}(B)).
\]
\end{lem}
\begin{proof}
Identical to the proof of \cite[Theorems 3.2, 3.6]{Sh1},
using (\ref{eqn:gammaToKoszul}) and (\ref{eqn:lambdaToTelescope}).
\end{proof}

\begin{prop}\label{prop:wideis}
Let $(A,\bar{\a}) \in \cdgcont$,
and let $P = (f,B,\mathbf{b})$ be a weakly proregular resolution of $(A,\bar{\a})$.
Then there are isomorphisms
\[
\mrm{R}\widehat{\Gamma}_{\bar{\a}}^P(-) \cong \mrm{L}(\Tau_A)^*(\mrm{R}\Gamma^P_{\bar{\a}}(-))
\]
and
\[
\mrm{L}\widehat{\Lambda}_{\bar{\a}}^P(-) \cong \mrm{R}(\Tau_A)^{\flat}(\mrm{L}\Lambda^P_{\bar{\a}}(-)).
\]
of functors $\cat{D}(A) \to \cat{D}(\mrm{L}\Lambda(A,\bar{\a}))$.
In particular, if $Q$ is another weakly proregular resolution of $(A,\bar{\a})$,
there are natural isomorphisms
\[
\mrm{R}\widehat{\Gamma}_{\bar{\a}}^P(-) \cong \mrm{R}\widehat{\Gamma}_{\bar{\a}}^Q(-)
\]
and
\[
\mrm{L}\widehat{\Lambda}_{\bar{\a}}^P(-) \cong \mrm{L}\widehat{\Lambda}_{\bar{\a}}^Q(-).
\]
\end{prop}
\begin{proof}
The first claim follows from Lemma \ref{lem:RWLW} and functoriality of of $\mrm{L}(-)^*$ and $\mrm{R}(-)^{\flat}$.
The second claim now follows from Corollary \ref{cor:uniq}.
\end{proof}
In view of this result, we denote this functors simply by
\[
\mrm{R}\widehat{\Gamma}_{\bar{\a}}, \mrm{L}\widehat{\Lambda}_{\bar{\a}} : \cat{D}(A) \to \cat{D}(\mrm{L}\Lambda(A,\bar{\a})).
\]
The next result is a version of (\ref{eqn:RWForget}) in this context.

\begin{prop}
For any $(A,\bar{\a}) \in \cdgcont$,
there are isomorphisms
\[
\mrm{R}(\Tau_A)^* \mrm{R}\widehat{\Gamma}_{\bar{\a}} \cong \mrm{R}\Gamma_{\bar{\a}}
\]
and
\[
\mrm{R}(\Tau_A)^* \mrm{L}\widehat{\Lambda}_{\bar{\a}} \cong \mrm{L}\Lambda_{\bar{\a}}
\]
of functors $\cat{D}(A) \to \cat{D}(A)$.
\end{prop}
\begin{proof}
Let $P = (f,B,\mathbf{b})$ be a weakly proregular resolution of $(A,\bar{\a})$.
By Theorem \ref{thm:main}, we have a commutative diagram
\[
\xymatrix{
P \ar[r]^f\ar[d]_{\tau_P} & A\ar[d]^{\Tau_A}\\
\Lambda_{\b}(P) \ar[r]_{\varphi_P} & \Lambda(A,\bar{\a})
}
\]
in $\ho(\cdg)$ in which the horizontal maps are isomorphisms.
Hence, by functoriality of $\mrm{R}(-)_*$, 
we have natural isomorphisms
\[
\mrm{R}(\Tau_A)^* \mrm{R}\widehat{\Gamma}^P_{\bar{\a}} \cong 
\mrm{R}(\tau_P \circ f^{-1})_* \mrm{R}(\varphi_P)_* \mrm{L}(\varphi_P)^* \mrm{R}\widehat{\Gamma}_{\b}(\mrm{R}f_*(-))
\]
Since $\phi_P$ is an isomorphism, 
we have an isomorphism of functors
\[
\mrm{R}(\phi_P)_* \mrm{L}(\phi_P)^* (-) \cong 1_{\Lambda_{\b}(P)},
\]
so that
\[
\mrm{R}(\Tau_A)^* \mrm{R}\widehat{\Gamma}^P_{\bar{\a}} \cong 
\mrm{R} (f^{-1})_* \mrm{R}(\tau_P)_* \mrm{R}\widehat{\Gamma}_{\b}(\mrm{R}f_*(-)) \cong
\mrm{R} (f^{-1})_* \mrm{R}\Gamma_{\b}(\mrm{R}f_*(-)) \cong \mrm{R}\Gamma_{\bar{\a}}(-).
\]
The second claim is proven similarly.
\end{proof}

\begin{prop}\label{prop:LWofHom}
Let $f:(A,\bar{\a}) \to (B,\bar{\b})$ be a map in $\cdgcont$ such that $\mrm{H}^0(f)(\bar{\a})\cdot \mrm{H}^0(B) = \bar{\b}$.
Then there is an isomorphism
\[
\mrm{L}\widehat{\Lambda}_{\bar{\b}} (\mrm{R}f^{\flat}(-)) \cong  \mrm{R}(\Tau_B \circ f)^{\flat}(\mrm{L}\Lambda_{\bar{\a}}(-))
\]
of functors $\cat{D}(A) \to \cat{D}(\mrm{L}\Lambda(B,\bar{\b}))$.
\end{prop}
\begin{proof}
Using Proposition \ref{prop:wideis},
we may write
\[
\mrm{L}\widehat{\Lambda}_{\bar{\b}} \mrm{R}f^{\flat}(-) \cong
\mrm{R}(\Tau_B)^{\flat} ( \mrm{L}\Lambda_{\bar{\b}} ( \mrm{R}f^{\flat}(-))).
\]
By Proposition \ref{prop:LLOfHom},
we have that
\[
\mrm{L}\Lambda_{\bar{\b}} ( \mrm{R}f^{\flat}(-)) \cong \mrm{R}f^{\flat}(\mrm{L}\Lambda_{\bar{\a}}(-)). 
\]
Hence,
\[
\mrm{L}\widehat{\Lambda}_{\bar{\b}} \mrm{R}f^{\flat}(-) \cong
\mrm{R}(\Tau_B)^{\flat} \circ \mrm{R}f^{\flat}(\mrm{L}\Lambda_{\bar{\a}}(-)) = \mrm{R}(\Tau_B \circ f)^{\flat}(\mrm{L}\Lambda_{\bar{\a}}(-)).
\]
\end{proof}

\begin{rem}
In less cryptic language, 
if all these morphisms in the homotopy category were honest morphisms of DG-rings, 
the above result says that there is an isomorphism
\[
\mrm{L}\widehat{\Lambda}_{\bar{\b}} \mrm{R}\opn{Hom}_A(B,-) \cong \mrm{R}\opn{Hom}_A(\widehat{B},\mrm{L}\Lambda_{\bar{\a}}(-)).
\]
\end{rem}

\begin{prop}\label{prop:ExtCompare}
Let $(A,\bar{\a}) \in \cdgcont$,
and let $f:\mrm{L}\Lambda(A,\bar{\a}) \to B$ be a map in $\cdg$.
Then there is an isomorphism
\[
\mrm{R}(f)^{\flat} (\mrm{L}\widehat{\Lambda}_{\bar{\a}}(-)) \cong \mrm{R}(\Tau_A \circ f)^{\flat}(\mrm{L}\Lambda_{\bar{\a}}(-))
\]
of functors $\cat{D}(A) \to \cat{D}(B)$.
\end{prop}
\begin{proof}
This follows immediately from Proposition \ref{prop:wideis}.
\end{proof}

\begin{rem}
Again, in less cryptic language, this result essentially says that there is a natural isomorphism
\[
\mrm{R}\opn{Hom}_{\widehat{A}}(B,\mrm{L}\widehat{\Lambda}_{\bar{\a}}(-)) \cong \mrm{R}\opn{Hom}_A(B,\mrm{L}\Lambda_{\bar{\a}}(-)),
\]
and thus, allows one to compare certain $\opn{Ext}$-modules over $A$ and $\widehat{A}$.
\end{rem}

The next results are dual to Propositions \ref{prop:LWofHom}, \ref{prop:ExtCompare}. We omit the similar proofs.

\begin{prop}\label{prop:RWofTen}
Let $f:(A,\bar{\a}) \to (B,\bar{\b})$ be a map in $\cdgcont$, such that $\mrm{H}^0(f)(\bar{\a}) \cdot \mrm{H}^0(B) = \bar{\b}$.
Then there is an isomorphism
\[
\mrm{R}\widehat{\Gamma}_{\bar{\b}} (\mrm{L}f^{*}(-)) \cong  \mrm{L}(\Tau_B \circ f)^{*}(\mrm{R}\Gamma_{\bar{\a}}(-))
\]
of functors $\cat{D}(A) \to \cat{D}(\mrm{L}\Lambda(B,\bar{\b}))$.
\end{prop}

\begin{prop}\label{prop:TorCompare}
Let $(A,\bar{\a}) \in \cdgcont$,
and let $f:\mrm{L}\Lambda(A,\bar{\a}) \to B$ be a map in $\cdg$.
Then there is an isomorphism
\[
\mrm{L}(f)^{*} (\mrm{R}\widehat{\Gamma}_{\bar{\a}}(-)) \cong \mrm{L}(\Tau_A \circ f)^{*}(\mrm{R}\Gamma_{\bar{\a}}(-))
\]
of functors $\cat{D}(A) \to \cat{D}(B)$.
\end{prop}

\begin{rem}
In case $A$ and $B$ are commutative rings,
and all ideals in question are weakly proregular, 
Propositions \ref{prop:LWofHom}, \ref{prop:ExtCompare}, \ref{prop:RWofTen} and \ref{prop:TorCompare} were essentially proved in 
\cite[Corollary 3.11]{Sh1}, \cite[Corollary 3.12]{Sh1}, \cite[Corollary 3.13]{Sh1}, \cite[Corollary 3.14]{Sh1}.
The above results show that with our notion of derived completion of DG-rings,
the weakly proregular hypothesis can be removed from the results of \cite{Sh1}.
See Section 6.1 below for more details.
\end{rem}

Given a commutative DG-ring $A$ such that $\mrm{H}^0(A)$ is a noetherian ring,
we denote by $\cat{D}_{\mrm{f}}(A)$ the full triangulated subcategory of $\cat{D}(A)$ consisting of DG-modules $M$ such that $\mrm{H}^n(M)$ is a finitely generated $\mrm{H}^0(A)$-module for all $n \in \mathbb{Z}$.

\begin{prop}\label{prop:comp-of-finite}
Let $A$ be a commutative DG-ring such that $\mrm{H}^0(A)$ is a noetherian ring,
and let $\bar{\a} \subseteq \mrm{H}^0(A)$ be an ideal.
Then for any $M \in \cat{D}_{\mrm{f}}(A)$,
we have that
\[
\mrm{L}\widehat{\Lambda}_{\bar{\a}}(M) \in \cat{D}_{\mrm{f}}(\mrm{L}\Lambda(A,\bar{\a})).
\]
\end{prop}
\begin{proof}
Since $\mrm{H}^0(A)$ is a noetherian ring,
the ideal $\bar{\a}$ is weakly proregular. 
Hence, by Proposition \ref{prop:lambda_of_wpr},
we have that 
\[
\mrm{H}^0(\mrm{L}\Lambda(A,\bar{\a})) \cong \Lambda_{\bar{\a}}(\mrm{H}^0(A)).
\]
In particular, the ring $\mrm{H}^0(\mrm{L}\Lambda(A,\bar{\a}))$ is noetherian,
so the notation $\cat{D}_{\mrm{f}}(\mrm{L}\Lambda(A,\bar{\a}))$ makes sense.
Since the functor $\mrm{L}\widehat{\Lambda}_{\bar{\a}}(-)$ has finite cohomological dimension, it is enough to prove the claim in the case where $M \in \cat{D}^{\mrm{b}}_{\mrm{f}}(A)$. By induction on $\amp(M)$ and using truncations, we may further assume that $\amp(M) = 0$. In that case, $M$ is isomorphic to a finitely generated $\mrm{H}^0(A)$-module, in which case the assertion is clear. Hence, the result.
\end{proof}

A commutative DG-ring $A$ is called noetherian if $\mrm{H}^0(A)$ is a noetherian ring,
and for each $i<0$, the $\mrm{H}^0(A)$-module $\mrm{H}^i(A)$ is finitely generated
(the terminology in \cite{Ye2} is cohomologically pseudo-noetherian).
If moreover $\mrm{H}^0(A)$ is a local noetherian ring, 
and $\bar{\m}$ is its maximal ideal, one says that $(A,\bar{\m})$ is a local noetehrian DG-ring. Proposition \ref{prop:comp-of-finite} now implies:

\begin{cor}
Let $A$ be a commutative noetherian DG-ring,
and let $\bar{\a} \subseteq \mrm{H}^0(A)$ be an ideal.
Then $\mrm{L}\Lambda(A,\bar{\a})$ is a commutative noetherian DG-ring.
If $(A,\bar{\m})$ is a local noetherian DG-ring,
then $(\mrm{L}\Lambda(A,\bar{\m}),\widehat{\bar{\m}})$ is a local noetherian DG-ring.
\end{cor}

\section{Applications}

In the previous sections, a very general theory of completion and torsion was developed.
In this final section we discuss some applications of this theory.

\subsection{Derived Hochschild cohomology and complete derived Hochschild cohomology}

Given a commutative ring $\k$,
and a commutative $\k$-algebra $A$,
we constructed in Section \ref{sec:dertenprod} the commutative DG-ring $A\otimes^{\mrm{L}}_{\k} A$, and the natural derived multiplication map $\Delta_A:A\otimes^{\mrm{L}}_{\k} A \to A$.
Suppose now that $\a \subseteq A$ is a finitely generated ideal,
and that $A$ is $\a$-adically complete,
and let
\[
\a^e := \a\otimes_{\k} A + A\otimes_{\k} \a \subseteq A\otimes_{\k} A = \mrm{H}^0(A\otimes^{\mrm{L}}_{\k} A).
\]
Applying the derived completion functor $\mrm{L}\Lambda$ to the map $\Delta_A:(A\otimes^{\mrm{L}}_{\k} A,\a^e) \to (A,\a)$, and using the fact that $A$ is $\a$-adically complete (and hence, by Proposition \ref{prop:com-is-dercom}, derived $\a$-adically complete), we obtain the derived complete multiplication map:
\[
\widehat{\Delta}_A:\mrm{L}\Lambda(A\otimes^{\mrm{L}}_{\k} A,\a^e) \to A.
\]
Let us denote $\mrm{L}\Lambda(A\otimes^{\mrm{L}}_{\k} A,\a^e)$ by $A\widehat{\otimes}^{\mrm{L}}_{\k} A$. 
Thus, it makes sense to denote the functor
\[
\widehat{\Delta}^\flat_A:\cat{D}(A\widehat{\otimes}^{\mrm{L}}_{\k} A) \to \cat{D}(A)
\]
by
\[
\mrm{R}\opn{Hom}_{A\widehat{\otimes}^{\mrm{L}}_{\k} A}(A,-),
\]
and call this functor the derived complete Hochschild cohomology functor (or derived complete Shukla cohomology).

\begin{thm}\label{thm:adicHoc}
Let $\k$ be a commutative ring,
let $A$ be a commutative $\k$-algebra,
let $\a \subseteq A$ be a finitely generated ideal,
and suppose that $A$ is $\a$-adically complete.
Let $M$ be an $\a$-adically complete $A$-module.
Then there is an isomorphism
\[
\mrm{R}\opn{Hom}_{A\otimes^{\mrm{L}}_{\k} A}(A,M) \cong
\mrm{R}\opn{Hom}_{A\widehat{\otimes}^{\mrm{L}}_{\k} A}(A,M)
\]
in $\cat{D}(A)$.
\end{thm}

Before proving this result, we shall need the next lemma which generalizes \cite[Lemma 4.2]{Sh1}:
\begin{lem}\label{lem:forAdicHoc}
Let $A$ be a commutative DG-ring,
let $\bar{\a} \subseteq \mrm{H}^0(A)$ be a finitely generated ideal,
let $B$ be a commutative ring,
and let $\varphi:\mrm{L}\Lambda(A,\bar{\a}) \to B$ be a map in $\ho(\cdg)$.
Let $\b \subseteq B$ be the finitely generated ideal generated by the image of $\widehat{\bar{\a}}$ in $B$, let $M$ be a $B$-module, and suppose that $M$ is $\b$-adically complete. Then there is an isomorphism
\[
\mrm{L}\widehat{\Lambda}_{\bar{\a}} \left(  \mrm{R}(\Tau_{\bar{\a}})_*(\mrm{R}\varphi_*(M)       ) \right) \cong \mrm{R}\varphi_*(M)   
\]
in $\cat{D}(\mrm{L}\Lambda(A,\bar{\a}))$,
and an isomorphism
\[
\mrm{L}\Lambda_{\bar{\a}} \left(  \mrm{R}(\Tau_{\bar{\a}})_*(\mrm{R}\varphi_*(M)       ) \right) \cong \mrm{R}(\Tau_{\bar{\a}})_*(\mrm{R}\varphi_*(M))
\]
in $\cat{D}(A)$.
\end{lem}
\begin{proof}
By replacing $A$ by a quasi-isomorphic DG-ring if necessary,
we may assume that there is a noetherian ring $\k$,
an ideal $\a\subseteq \k$, and a flat map $\k \to A^0$,
such that $\a \cdot \mrm{H}^0(A) = \bar{\a}$,
and that moreover that $A$ is K-flat over $A^0$. 
In particular, $\a\cdot A^0$ is a weakly proregular ideal.
It follows that $\mrm{L}\Lambda(A,\bar{\a}) \cong \Lambda_{\a}(A)$.
Let $N = \mrm{R}\varphi_*(M) \in \cat{D}(\mrm{L}\Lambda(A,\bar{\a}))$,
and let $K = \mrm{R}(\Tau_{\bar{\a}})_*(N) \in \cat{D}(A)$.
These are both DG-modules that are concentrated in degree $0$.
Let $P \to K$ be a K-flat resolution over $A$.
Then by definition, 
\[
\mrm{L}\widehat{\Lambda}_{\bar{\a}} \left(  \mrm{R}(\Tau_{\bar{\a}})_*(\mrm{R}\varphi_*(M)       )\right) = \widehat{\Lambda}_{\a}(P).
\]
Since $P$ is K-flat over $A$, $A$ is K-flat over $A^0$, and $A^0$ is flat over $\k$,
we see that as also as a complex of $\k$-modules, $P$ is a K-flat resolution of the $\k$-module $M$. Then, it follows from \cite[Theorem 1.21]{PSY2} that,
at least as a complex of $\k$-modules, $\widehat{\Lambda}_{\a}(P)$ is concentrated in degree $0$, and there it is isomorphic to $M$. It follows that the $A$-linear natural map $P \to \Lambda_{\a}(P)$ is an isomorphism in $\cat{D}(A)$.
Thus, considered as $\mrm{H}^0(A)$-modules,
the modules $\mrm{H}^0(N)$ and 
\[
\mrm{H}^0\left(\mrm{L}\widehat{\Lambda}_{\bar{\a}} \left(  \mrm{R}(\Tau_{\bar{\a}})_*(\mrm{R}\varphi_*(M)       )\right)\right)
\]
are isomorphic. 
Since both of these modules are $\bar{\a}$-adically complete, we deduce that they are also isomorphic over $\Lambda_{\bar{\a}}(\mrm{H}^0(A))$, 
and hence, also over $\mrm{H}^0(\Lambda_{\a}(A))$ (which in general may be different, but surjects to the ring $\Lambda_{\bar{\a}}(\mrm{H}^0(A))$).
In particular, there is an isomorphism 
\[
\mrm{L}\widehat{\Lambda}_{\bar{\a}} \left(  \mrm{R}(\Tau_{\bar{\a}})_*(\mrm{R}\varphi_*(M)       ) \right) \cong \mrm{R}\varphi_*(M)   
\]
in $\cat{D}(\mrm{L}\Lambda(A,\bar{\a}))$, as claimed.
We omit the proof of the second claim, as it is similar, but easier.
\end{proof}

We now prove Theorem \ref{thm:adicHoc}:
\begin{proof}
By definition, we have:
\[
\mrm{R}\opn{Hom}_{A\widehat{\otimes}^{\mrm{L}}_{\k} A}(A,M)
=
\mrm{R}(\widehat{\Delta}_A)^{\flat}\left(\mrm{R}(\widehat{\Delta}_A)_*(M))\right)
\]
By the first claim of Lemma \ref{lem:forAdicHoc}, we have an isomorphism
\[
\mrm{R}(\widehat{\Delta}_A)^{\flat}\left(\mrm{R}(\widehat{\Delta}_A)_*(M)\right)
\cong
\mrm{R}(\widehat{\Delta}_A)^{\flat}\left(
\mrm{L}\widehat{\Lambda}_{\a^e}\left( \mrm{R}(\Tau_{\a^e})_*(\mrm{R}(\widehat{\Delta}_A)_*(M))       \right)\right)
\]
in $\cat{D}(A)$. 
Using Lemma \ref{lem:RWLW}, we have that
\begin{eqnarray}
\mrm{R}(\widehat{\Delta}_A)^{\flat}\left(
\mrm{L}\widehat{\Lambda}_{\a^e}\left( \mrm{R}(\Tau_{\a^e})_*(\mrm{R}(\widehat{\Delta}_A)_*(M))       \right) \right) \cong\nonumber\\
\mrm{R}(\widehat{\Delta}_A))^{\flat}\left(
\mrm{R}(\Tau_{\a^e})^{\flat}\left(\mrm{L}\Lambda_{\a^e} 
(
\mrm{R}(\Tau_{\a^e})_*(\mrm{R}(\widehat{\Delta}_A)_*(M))
)
\right)
\right)\nonumber
\end{eqnarray}
By the second claim of Lemma \ref{lem:forAdicHoc},
\begin{eqnarray}
\mrm{R}(\widehat{\Delta}_A))^{\flat}\left(
\mrm{R}(\Tau_{\a^e})^{\flat}\left(\mrm{L}\Lambda_{\a^e} 
(
\mrm{R}(\Tau_{\a^e})_*(\mrm{R}(\widehat{\Delta}_A)_*(M))
)
\right)
\right)
\cong\nonumber\\
\mrm{R}(\widehat{\Delta}_A))^{\flat}\left(
\mrm{R}(\Tau_{\a^e})^{\flat}\left(
\mrm{R}(\Tau_{\a^e})_*(\mrm{R}(\widehat{\Delta}_A)_*(M))
\right)
\right)\nonumber
\end{eqnarray}
The hom-tensor adjunction 
\[
\mrm{R}\opn{Hom}_{A\widehat{\otimes}^{\mrm{L}}_{\k} A}(A,
\mrm{R}\opn{Hom}_{A\otimes^{\mrm{L}}_{\k} A}(A\widehat{\otimes}^{\mrm{L}}_{\k} A,-))
\cong
\mrm{R}\opn{Hom}_{A\otimes^{\mrm{L}}_{\k} A}(A,-)
\]
implies that
\[
\mrm{R}(\widehat{\Delta}_A))^{\flat}\left(
\mrm{R}(\Tau_{\a^e})^{\flat}\left(-\right)\right) \cong
\mrm{R}(\Delta_A)^{\flat}\left( - \right)
\]
Similarly, as the composition of the forgetful functors 
\[
\cat{D}(A) \to \cat{D}(A\widehat{\otimes}^{\mrm{L}}_{\k} A) \to \cat{D}(A\otimes^{\mrm{L}}_{\k} A) 
\]
is equal to the forgetful functor $\cat{D}(A) \to \cat{D}(A\otimes^{\mrm{L}}_{\k} A)$,
we see that
\[
\mrm{R}(\Tau_{\a^e})_*\left(\mrm{R}(\widehat{\Delta}_A)_*\left(-\right)\right) \cong \mrm{R}\Delta_A^*(-).
\]
These two facts imply that
\begin{eqnarray}
\mrm{R}(\widehat{\Delta}_A))^{\flat}\left(
\mrm{R}(\Tau_{\a^e})^{\flat}\left(
\mrm{R}(\Tau_{\a^e})_*(\mrm{R}(\widehat{\Delta}_A)_*(M))
\right)
\right) \cong\nonumber\\
\mrm{R}(\Delta_A)^{\flat}\left( 
\mrm{R}\Delta_A^*(M)
\right),\nonumber
\end{eqnarray}
and the latter is exactly
\[
\mrm{R}\opn{Hom}_{A\otimes^{\mrm{L}}_{\k} A}(A,M),
\]
which proves the claim.
\end{proof}

As an application of this result, we now deduce the finiteness of the derived Hochschild cohomology modules of affine formal schemes:

\begin{cor}
Let $\k$ be a commutative noetherian ring,
let $A$ be a noetherian $\k$-algebra,
and let $\a\subseteq A$ be an ideal such that $A$ is $\a$-adically complete,
and such that $A/\a$ is a finitely generated $\k$-algebra\footnote{This finiteness condition is an analogue of the finite type condition in the theory of formal schemes. Such a map is called formally of finite type in \cite{Ye0}, and of pseudo finite type in \cite{AJL3}.}. Let $M$ be a finitely generated $A$-module.
Then the derived Hochschild cohomology $A$-modules
\[
\operatorname{Ext}^n_{A\otimes^{\mrm{L}}_{\k} A}(A,M)
\]
are finitely generated for all $n \in \mathbb{Z}$.
\end{cor}
\begin{proof}
Using Theorem \ref{thm:adicHoc}, and because of \cite[Theorem 2.13]{Ye2}, it is enough to show that the commutative DG-ring $A\widehat{\otimes}^{\mrm{L}}_{\k} A$ is noetherian.
As explained in the proof of \cite[Proposition 1.4]{Ye0}, it follows from \cite[Proposition 1.1]{Ye0} that the map $\k \to A$ factors as $\k \to \k[x_1,\dots,x_m][[y_1,\dots,y_k]] \to A$, for some $m,k \in \mathbb{N}$, such that the map 
\[
\k[x_1,\dots,x_m][[y_1,\dots,y_k]] \to A
\]
is surjective, and moreover, $(y_1,\dots,y_k)\cdot A = \a$.
Let us set $B := \k[x_1,\dots,x_m][[y_1,\dots,y_k]]$. 
Of course, $B$ is a noetherian ring, and moreover, it is flat over $\k$.
By \cite[Proposition 2.1.10]{Av}, the map $B \to A$ factors as
$B \to \til{A} \to A$,
where $\til{A}$ is a commutative DG-ring, $\til{A} \to A$ is a quasi-isomorphism,
$\til{A}^0 = B$, and for each $i<0$, $\til{A}^i$ is a finitely generated free $\til{A}^0$-module. In particular $\til{A}$ is K-flat over $B$, and hence, also over $\k$.
Hence, $A\otimes^{\mrm{L}}_{\k} A \cong \til{A} \otimes_{\k} \til{A}$.
Letting $\mathbf{b}^e = (y_1\otimes_{\k} 1, \dots y_k\otimes_{\k} 1, 1\otimes_{\k} y_1,\dots, 1\otimes_{\k} y_k)$, we see that $(\til{A}\otimes_{\k} \til{A},\mathbf{b}^e)$ is a weakly proregular DG-ring, and it follows that
\[
A\widehat{\otimes}^{\mrm{L}}_{\k} A \cong \Lambda_{\mathbf{b}^e}(\til{A}\otimes_{\k} \til{A}).
\]
The latter is easy to compute, as
\[
(\Lambda_{\mathbf{b}^e}(\til{A}\otimes_{\k} \til{A}))^0 = \k[x_1,\dots,x_{2m}][[y_1,\dots,y_{2k}]],
\]
and for each $i<0$, $(\Lambda_{\mathbf{b}^e}(\til{A}\otimes_{\k} \til{A}))^i$ is a finitely generated free $(\Lambda_{\mathbf{b}^e}(\til{A}\otimes_{\k} \til{A}))^0$-module.
Since $(\Lambda_{\mathbf{b}^e}(\til{A}\otimes_{\k} \til{A}))^0$ is noetherian, it follows that $A\widehat{\otimes}^{\mrm{L}}_{\k} A$ is a noetherian DG-ring, proving the claim.
\end{proof}

\subsection{Further applications}

Since local cohomology and adic completion are fundamental operations in commutative algebra,
we expect that, similarly, local cohomology and derived completion will have a deep theory and many applications in derived commutative algebra. Here, we give some examples of applications of this paper that we obtained since the first version of this paper was written.

\subsubsection{Endomorphism rings of indecomposable injectives}

If $A$ is a commutative noetherian ring,
the Matlis classification (\cite{Ma}) of injective modules over $A$ implies that isomorphism classes of indecomposable injective $A$-modules correspond bijectively with $\opn{Spec}(A)$.
Given $\p \in \opn{Spec}(A)$, letting $I_{\p}$ be the injective hull of $A/\p$,
Matlis proved that its endomorphism ring $\opn{Hom}_A(I_p,I_p)$ is isomorphic to $\Lambda_{\p}(A_{\p})$. 
In \cite{Sh4}, using the results of this paper, we generalized this result to derived commutative algebra: given a commutative noetherian DG-ring $A$,
for any $\bar{\p} \in \opn{Spec}(\mrm{H}^0(A))$, we constructed a DG-module $I_{\bar{\p}}$ which is an generalization of $I_{\p}$.  
Then, in \cite[Theorem 7.22]{Sh4}, we proved that its derived endomorphism DG-ring
\[
\mrm{R}\opn{Hom}_A(I_{\bar{\p}},I_{\bar{\p}})
\]
is isomorphic to the derived completion $\mrm{L}\Lambda(A_{\bar{\p}},\bar{\p})$.

\subsubsection{Injective dimension of torsion and flat dimension of completion}

It is a basic corollary of the Matlis structure theory of injectives that if $A$ is a commutative noetherian ring, $\a \subseteq A$ is an ideal,
and $I$ is an injective $A$-module, then $\Gamma_{\a}(I)$ is also an injective $A$-module.
More generally, if $M$ is a complex of $A$-modules with $\opn{inj}\dim_A(M) < \infty$,
then $\opn{inj}\dim_A(\mrm{R}\Gamma_{\a}(M)) \le \opn{inj}\dim_A(M)$.
In \cite[Theorem 3.5]{Sh35}, we generalized this fact to commutative noetherian DG-rings $A$, and ideals $\bar{\a} \subseteq \mrm{H}^0(A)$. 
A similar result (\cite[Theorem 4.3]{Sh35}) was shown there about the flat dimension of derived completion. From it, we deducd that if $A$ is a commutative noetherian DG-ring, with $\mrm{H}(A)$ bounded, then the derived completion map $\Tau_{\bar{\a}}:A \to \mrm{L}\Lambda(A,\bar{\a})$ has flat dimension $0$, an analogue of the classical fact that the adic completion of a commutative noetherian ring is flat over it.

\subsubsection{Local duality}

Grothendieck's local duality theorem is a fundamental result in the theory of local cohomology over local rings.
In \cite[Theorem 7.26]{Sh4}, using the results of this paper, 
we showed that the local duality theorem generalizes to commutative noetherian local DG-rings $(A,\bar{\m})$, with respect to the functor $\mrm{R}\Gamma_{\bar{\m}}$ from Section \ref{sec:RL}.

\subsubsection{Existence of dualizing DG-modules over complete local DG-rings}

A basic technique in commutative algebra employed in the study of local rings $(A,\m)$ is to pass to their completion $\Lambda_{\m}(A)$, which is faithfully flat over $A$, and then use the Cohen structure theorem for complete local rings, which says that they are quotients of complete regular local rings.
In particular, complete local rings posses dualizing complexes, and this fact is of great importance in their study. Given a commutative noetherian DG-ring $A$,
and given $\bar{\p} \in \opn{Spec}(\mrm{H}^0(A))$, one can localize $A$ at $\bar{\p}$ and obtain a commutative noetherian local DG-ring $(A_{\bar{\p}},\bar{\p})$.
Its derived completion $(\mrm{L}\Lambda(A_{\bar{\p}},\bar{\p}),\widehat{\bar{\p}})$ is complete local DG-ring. In \cite[Proposition 7.21]{Sh4}, we proved that complete local DG-rings have dualizing DG-modules.

\textbf{Acknowledgments.}

The author is grateful to the anonymous referee for his suggestions that helped improving this manuscript.


\end{document}